\documentclass[reqno]{amsart}
\usepackage{enumerate}
\usepackage{bbm}
\usepackage{tikz}

\usepackage{amsmath,amsthm,amscd,amssymb}
\usepackage{latexsym}
\usepackage{color}
\usepackage{calrsfs}
\usepackage{todonotes}
\usetikzlibrary{decorations.pathreplacing}

\numberwithin{equation}{section}

\newtheorem{theorem}{Theorem}[section]

\newtheorem{cor}[theorem]{Corollary}
\newtheorem{lemma}[theorem]{Lemma}
\newtheorem{prop}[theorem]{Proposition}
\newtheorem{defi}[theorem]{Definition}

\newtheorem{ex}[theorem]{Example}

\newcommand{\N}{{\mathbb N}}
\newcommand{\X}{{\mathcal X}}

\newcommand{\E}{{\mathcal E}}
\renewcommand{\P}{{\mathcal P}}

\begin{document}

\title[Stick-breaking, clumping, and Markov chains]
{Stick-breaking processes, clumping, and Markov chain occupation laws}

\author{Zach Dietz, William Lippitt, Sunder Sethuraman}
\date{}

\address{\noindent Zach Dietz:
\newline
e-mail:  \rm \texttt{zedietz@yahoo.com}
}

\address{\noindent William Lippitt: Department of Mathematics, University of Arizona,
  Tucson, AZ  85721
\newline
e-mail:  \rm \texttt{wlippitt@math.arizona.edu}
}

\address{\noindent Sunder Sethuraman: Department of Mathematics, University of Arizona,
  Tucson, AZ  85721
\newline
e-mail:  \rm \texttt{sethuram@math.arizona.edu}
}

\begin{abstract}
 We consider the connections among `clumped' residual allocation models (RAMs), a general class of stick-breaking processes including Dirichlet processes,
and the occupation laws of certain discrete space time-inhomogeneous Markov chains related to simulated annealing and other applications.  An intermediate structure is introduced in a given RAM, where proportions between successive indices in a list are added or clumped together to form another RAM.  
In particular, when the initial RAM is a Griffiths-Engen-McCloskey (GEM) sequence and the indices are given by the random times that an auxiliary Markov chain jumps away from its current state, 
the joint law of the intermediate RAM and the locations visited in the sojourns is given in terms of a `disordered' GEM sequence, and an induced Markov chain.  
Through this joint law, we identify a large class of `stick breaking' processes as the limits of empirical occupation measures for associated time-inhomogeneous Markov chains.
\end{abstract}

\subjclass[2010]{60G57, 60E99, 60J10}

\keywords{residual allocation model, RAM, GEM, Dirichlet, inhomogeneous, Markov, stick breaking, occupation, empirical, clumping}

\maketitle

\section{Introduction and summary}

In this article, we introduce an intermediate `clumped' structure in residual allocation models of apportionment of a resource, such as Griffiths-Engen-McCloskey (GEM) models.  Although this intermediate structure is perhaps of its own interest, through it, we identify the empirical occupation law limits in a class of time-inhomogeneous discrete space Markov chains, associated with simulated annealing and other applications, as new types of stick-breaking processes built from Markovian samples, including Dirichlet processes.  On the one hand, GEM models and Dirichlet processes have wide application in population genetics, ecology, combinatorial stochastic processes, and Bayesian nonparametric statistics; see books and surveys \cite{Crane1}, \cite{Crane2},  \cite{Ghosal_vandeVaart}, \cite{Ghosh_Rama}, \cite{Hjort}, \cite{Pitman} and references therein. On the other hand, the time-inhomogeneous Markov chains that we consider are stylized models of simulated annealing and Gibbs samplers or types of mRNA dynamics; see \cite{Bouguet}, \cite{DS}, \cite{Englander}, \cite{Gantert}, \cite{Herbach},  \cite{Winkler}.  In a sense, one purpose of the paper is to observe a perhaps unexpected connection between these apriori different objects.

We now discuss some of the relevant background on GEM and Dirichlet measures, and time-inhomogeneous Markov chains, before turning to an informal discussion of our results on the intermediate structure in GEM sequences and their connections with the occupation laws of the Markov chains.

\subsection
{GEM and Dirichlet measures} Consider the infinite-dimensional simplex $\Delta_\infty$
of all
all discrete (probability) distributions on $\N=\{1,2,\ldots\}$. 
A residual allocation model (RAM) is a distribution on $\Delta_\infty$, introduced in the 1940's \cite{Halmos} as a means to address problems of apportionment:
Let $\{X_n\}_{n\geq 1}$ be independent $[0,1]$-valued random variables, called `residual fractions'.
Consider the associated process $\langle P_n: n\geq 1\rangle \in [0,1]^{\N}$, given by $P_1=X_1$ and
$$P_n = \left(1 -
\sum_{j=1}^{n-1}P_j\right)X_n = (1-X_1)\cdots (1-X_{n-1})X_n\ \ {\rm for \ }n\geq 2;$$
see Lemma \ref{oneidentity} for the induction leading to the last equality. If $\sum_{n\geq 1}P_n \stackrel{a.s.}{=} 1$, the distribution $\langle P_n: n\geq 1\rangle \in \Delta_\infty$ is the associated RAM.  In general, $\langle P_n: n\geq 1\rangle$ need not sum to $1$ for a given realization.  We note a simple condition equivalent to $\sum_{n\geq 1}P_n\stackrel{a.s.}{=}1$ is that $\prod_{j=1}^\infty(1-X_j)\stackrel{a.s.}{=}0$, the case for nontrivial, independent, identically distributed (iid) fractions (cf. Lemma \ref{oneidentity}).

The RAM when the fractions $\{X_n\}_{n\geq 1}$ are iid Beta$(1,\theta)$ random variables is the well-known Griffiths-Engen-McCloskey GEM$(\theta)$ model. 
There are many characterizations and studies of the GEM sequence and its variants in recent years.  For instance, the GEM model is the unique RAM with iid fractions that is invariant in law under size-biased permutation.   Also, the GEM sequence is the unique invariant measure of `split and merge' dynamics.  In addition, there are important connections with Poisson-Dirichlet models.  See for instance, among others, \cite{Arratia}, \cite{Arratia2}, \cite{Diaconis}, \cite{Engen},  \cite{Gnedin}, \cite{Hoppe}, \cite{Kingman1}, \cite{Kingman2}, \cite{McCloskey}, \cite{Patil_Taillie}, \cite{Pitman1}, \cite{Pitman2}, \cite{Pitman_Y}, and references therein.

Moreover, the GEM 
sequence is a fundamental building block of Dirichlet processes, which often serve as a measure on priors in Bayesian nonparametric statistics \cite{Ghosal_vandeVaart},  \cite{Ghosh_Rama}.  With respect to a measurable space $(\mathcal{X}, \mathcal{B})$, consider the space of probability measures $\mathbb{P}_{\mathcal{X}}$ endowed with $\sigma$-field generated by the sets $\{P: P(A)<r\}$ for $A\in \mathcal{B}$ and $r>0$.  We say that $D$ is a random probability sample from the Dirichlet process, with `parameters' $\theta>0$ and probability measure $\mu$ on $\mathcal{X}$, if for any finite partition $\{A_i\}_{i=1}^m$ that the vector $\langle D(A_1),\ldots, D(A_m)\rangle$ has the Dirichlet distribution with parameters $\langle\theta\mu(A_i): 1\leq i\leq m\rangle$.  

The `stick breaking' representation of the Dirichlet process with parameters $(\theta, \mu)$, in terms of a GEM$(\theta)$ sequence $\langle P_i: i\geq 1\rangle$, and an independent sequence of iid random variables $\{Z_i\}_{i\geq 1}$ with common distribution $\mu$, is given by
\begin{equation}
\label{Dirichlet_intro}
D(\cdot; \theta, \mu)= \sum_{i=1}^\infty P_i \delta_{Z_i}(\cdot).
\end{equation}
There is a large literature on Dirichlet processes stemming from the seminal works \cite{Blackwell_MacQueen}, \cite{Ferguson}.  See \cite{Pitman2}, \cite{JSethuraman} with respect to the `stick breaking' construction, and books \cite{Ghosal_vandeVaart}, \cite{Ghosh_Rama}, \cite{Muller}, \cite{Pitman} for more on their history, other representations including that with respect to the `Chinese restaurant process', and their use in practice.

In this article, we will concentrate on discrete spaces $\mathcal{X}\subset \N$, that is those composed of either a finite or a countably infinite number of elements.   We note, when $\mathcal{X}=\{1,\ldots k\}$ is finite, $\mu = \langle \mu(1),\ldots, \mu(k)\rangle$ and and $A_i = \{i\}$ for $1\leq i\leq k$, the property that $\langle D(A_1),\ldots, D(A_k)\rangle$ is given by a Dirichlet distribution was first stated in a population genetics context in \cite{Donnelly_Tavare}; see also 
\cite{Hirth}.

\subsection{Time-inhomogeneous Markov chains} Let $G$ be a generator kernel on $\mathcal{X}$, that is $G_{i,j} \geq 0$ for $i\neq j\in\mathcal{X}$, and $G_{i,i} = -\sum_{j\neq i} G_{i,j}$.  Suppose the entries of $G$ are suitably bounded so that the kernel 
\begin{equation}
\label{Q_n}
K_n = I + \frac{G}{n}
\end{equation}
 is a stochastic kernel for all $n$ large enough, and set $K_n=I$ otherwise.  Let $\{T_n\}_{n\geq 1}$ be the time-inhomogeneous Markov chain on the discrete space $\mathcal{X}$ associated to kernels $\{K_n\}_{n\geq 1}$.  Consider $G$ without zero rows.  Then, every point in $\mathcal{X}$ represents a valley from which the chain rarely but almost surely exits to enter another point valley.  In this way, a certain `landscape' is explored.  The chain can be considered as a simplified model of simulated annealing or metastability (cf. \cite{Bovier}, \cite{Gantert}, \cite{Landim}, \cite{OV}, \cite{Winkler}).  From another view, continuous-time variants of such inhomogeneous chains have been used in the modeling of certain mRNA dynamics \cite{Herbach}.

Interestingly, for finite $\mathcal{X}$, it was noted in \cite{Gantert} and \cite{Winkler} that the sample means of these chains do not converge a.s. or in probability, as would be the case for a homogeneous Markov chain.  For generators $G$ without zero entries, weak convergence to an empirical occupation law
\begin{equation}
\label{DS_limits}
\nu = \lim_{n\rightarrow\infty}  \left\langle \frac{1}{n}\sum_{j=1}^n \delta_{T_j}(i): i\in \mathcal{X}\right\rangle,
\end{equation}
was identified by computing its moments in \cite{DS}.  Curiously, when $G$ is of the form $G= \theta(Q - I)$ for $\theta>0$ and $Q$ a stochastic matrix with constant rows $\mu$,  it was also shown that $\nu$ is a Dirichlet distribution with parameters $\{\theta \mu(i)\}_{i=1}^k$ by matching the moments.    
Similar occupation laws were also derived in the continuous-time mRNA model in \cite{Herbach} 
as the stationary distributions of a promoter process on $k$ states, influencing levels of mRNA production.  

In this context, part of our motivation is to understand this limit and its generalizations more constructively (Theorem \ref{occ cor}).

\subsection{Clumped structure and generalized `stick-breaking' processes}  We now describe a class of generalized stick-breaking processes.  
Let $\langle P_i: i\geq 1\rangle$ be a GEM$(\theta)$ sequence and, to be focused, let $\{T_i'\}_{i\ge 1}$ be an independent Markov chain with irreducible, recurrent transition kernel $Q$ on a discrete space $\mathcal{X}$ with initial distribution $\pi$, although we also consider more general Markov chains, not necessarily irreducible or composed only of recurrent states, in several of our results.

Another motivation of ours is to understand the random measures
\begin{equation}
\label{generalF}
\nu(\cdot; \theta, \mu, Q) = \sum_{i=1}^\infty P_i \delta_{T_i'}(\cdot),
\end{equation}
seen as a natural generalization of stick-breaking representation of the Dirichlet process, with respect to Markovian samples $\{T'_i\}_{i\geq 1}$ instead of the iid ones in \eqref{Dirichlet_intro}.  

In general, $\nu$ is not exchangeable in the sense that the GEM sequence $\langle P_i: i\geq 1\rangle$ may not be replaced by an arbitrary permutation without changing the measure.  In contrast, when $\{T_i'\}_{i\geq 1}$ is iid and $\nu$ is the Dirichlet process, such an exchangeability property holds; for example, the Poisson-Dirichlet order statistics $\langle \hat P_i: i\geq 1\rangle$ of $\langle P_i: i\geq 1\rangle$ may be used instead without changing the Dirichlet process (cf. \cite{Pitman2}). 
We also note that other generalizations of Dirichlet processes have been considered, among them, Polya tree \cite{Lavine}, Pitman-Yor \cite{Pitman2}, \cite{Pitman_Yor}, and Beta processes \cite{BJP}.

We now introduce a clumped intermediate structure which will help analyze $\nu$.  Suppose $\{V_i\}_{i\geq 1}$ are the times when the Markov chain jumps to a different state with the convention $V_1=1$.  In particular, `skip-repetition' is allowed:  The chain can begin in state $x$, jump to $y\neq x$ at time $V_2$, and then may jump back at time $V_3$ into state $x$.  We note that these times are not only those times when a state is observed for the first time, as used in the definition of size-biased permutations.

Consider $R_i = \sum_{j=V_i}^{V_{i+1}-1} P_j$ for $i\geq 1$.   We show that (cf. Theorems \ref{RAM} and \ref{Gem to MCcGEM}), conditional on the locations $\{Y_i = T_{V_i}'\}_{i\geq 1}$, the sequence $\langle R_i: i\geq 1\rangle$ is a RAM where the associated fractions are Beta$\big(1,\theta (1-Q_{Y_i,Y_i})\big)$ for $i\geq 1$, a sort of `disordered' GEM.  Also, the law of $\{Y_i\}_{i\geq 1}$ can be computed as another Markov chain on $\mathcal{X}$ with a transition kernel found in terms of $Q$.  We will call the joint law of $\big(\langle R_i: i\geq 1\rangle, \{Y_i\}_{i\geq 1}\big)$ as a type of Markov Chain conditional GEM, or `MCcGEM' distribution.

In terms of the clumped intermediate structure, we see that 
\begin{equation}
\label{clumpedF}
\nu(\cdot) = \sum_{i=1}^\infty R_i \delta_{Y_i}(\cdot).
\end{equation}
This representation will allow us to identify $\nu$ as the limit of occupation laws of a matched time-inhomogeneous Markov chain (Theorems \ref{occ cor}, \ref{occ cor2}).

  We will also see that $\nu$ satisfies a `self-similarity' equation (cf. Theorem \ref{selfsimgem}), uniquely characterizing its distribution.  This equation is reminiscent of the regenerative structure present in `stick-breaking'  \cite{JSethuraman}, in integral constructions of the Dirichlet processs \cite{Last}, \cite{SL}, and in other related settings \cite{Gnedin_Pitman1}, \cite{Gnedin_Pitman2}.  
  
  Moreover, when $\X$ is finite, we discuss the joint moments of the distribution in Theorem \ref{moments_thm}.  Although a formula for the moments is given in \cite{DS}, the description in Theorem \ref{moments_thm} is more detailed, allowing identification of the marginal distributions as Beta products (cf. Theorem \ref{momentchain} and Corollary \ref{marginalmom}).

\subsection{Occupation laws of time-inhomogeneous Markov chains}\label{tikpic}
With respect to the time-inhomogeneous Markov chain ${\bf T}=\{T_n\}_{n\geq 1}$ with kernels $\{K_n\}_{n\geq 1}$ \eqref{Q_n}, starting from initial distribution $\mu$, consider the random empirical occupation measure on $\mathcal{X}$,
$$\nu_n(\cdot) = \frac{1}{n}\sum_{j=1}^n \delta_{T_i}(\cdot).$$
To connect with the intermediate clumping structure from the previous section, we will again implement a clumping procedure, this time to investigate local occupations, or clumped occupations, of the empirical measure of ${\bf T}$ up to time $n$. 

However, in a Markov chain with kernels $\{K_n\}_{n\geq1}$, later clumps of the chain are typically larger than earlier clumps. To keep the clump sizes from tending to zero after normalization, we consider the clumps in reverse chronological order, starting from time $n$, so that the clumped occupations converge nontrivially in distribution.

Formally, let $1=V_1<V_2<\cdots$ be the successive times when the Markov chain changes state, and let $N_n=\min\{i: V_i>n\}$.  Going backwards from time $n$, let $\tau_{n,1}$ be the length $n+1-V_{N_n-1}$ of the last visit to state $Y_{n,1}= T_{V_{N_n-1}}$, $\tau_{n,2}$ be the length $V_{N_n-1}- V_{N_n -2}$ of the visit to state $Y_{n,2} = T_{V_{N_n -2}}$, and $\tau_{n,k}$ be the length $V_{N_n -(k-1)} - V_{N_n -k}$ of the visit to $Y_{n,k} =T_{V_{N_n - k}}$ for $1< k< N_n$.  Let also $\tau_{n,k}=0$ and $Y_{n,k}=T_1$ for $k\geq N_n$. In addition, define $P_{n,k} = \tau_{n,k}/n$ for $k\geq 1$.

The figure below depicts, in a realization, the clumping boundaries $V_j$ marked in forward times, and the lengths of local occupations $\tau_{n,j}=nP_{n,j}$ given backwards in time starting from time $n$.
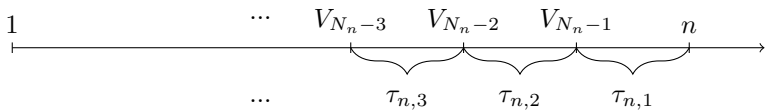
\begin{center}
\begin{tikzpicture}
\draw[->] (0,0) -- (10,0);
\node at (3.3,.4){...};
\node at (3.3,-.7){...};
\draw (0,-2pt) -- (0,2pt)node[anchor=south] {1};
\draw (4.5,-2pt) -- (4.5,2pt) node[anchor=south]{$V_{N_n-3}$};
\draw (6,-2pt) -- (6,2pt) node[anchor=south]{$V_{N_n-2}$};
\draw (7.5,-2pt) -- (7.5,2pt) node[anchor=south]{$V_{N_n-1}$};
\draw (9,-2pt) -- (9,2pt) node[anchor=south]{$n$};
\draw [decorate,decoration={brace,amplitude=10pt}]
(9,0) -- (7.5,0) node [midway,yshift=-0.7cm]{$\tau_{n,1}$};
\draw [decorate,decoration={brace,amplitude=10pt}]
(7.5,0) -- (6,0) node [midway,yshift=-0.7cm]{$\tau_{n,2}$};
\draw [decorate,decoration={brace,amplitude=10pt}]
(6,0) -- (4.5,0) node [midway,yshift=-0.7cm]{$\tau_{n,3}$};
\end{tikzpicture}
\end{center}

Then, $\nu_n$ is written as
$$\nu_n(\cdot) = \sum_{j=1}^{N_n-1} P_{n,j} \delta_{Y_{n,j}}(\cdot) = \sum_{j=1}^\infty P_{n,j} \delta_{Y_{n,j}}(\cdot).$$

We show (cf. Theorem \ref{inhom to MCcGEM}), for generators $G$ satisfying natural conditions, conditionally on the values $\{Y_{n,j}\}_{j\geq 1}$, that the distributions of $\langle P_{n,j}: j\geq 1\rangle$ converge, as $n\rightarrow\infty$, to a disordered GEM $\langle P^+_j: j\geq 1\rangle$ with parameters given in terms of $G$ and $\mu$.  Also, $\{Y_{n,j}\}_{j\geq 1}$ converges, as $n\rightarrow\infty$, to a homogeneous Markov chain $\{Y_j\}_{j\geq 1}$, with transition kernel in terms of $G$ and $\mu$.  In particular, the joint law of $\langle P_{n,j}: j\geq 1\rangle$ and $\{Y_{n,j}\}_{j\geq 1}$ converges, as $n\rightarrow\infty$, to a Markov Chain conditional GEM distribution, denoted as the MCcGEM$(G)$ distribution with respect to $\mu$. 

In Theorem \ref{occ cor}, we will then be able to show that $\nu_n$ converges to a random measure $\nu$ given in terms of $\langle P^+_j: j\geq 1\rangle$ and $\{Y_j\}_{j\geq 1}$ either in `stick-breaking' or `clumped' forms \eqref{generalF}, \eqref{clumpedF}.
In particular, when $G = \theta(Q-I)$ where $Q$ is a constant stochastic matrix with identical rows $\mu$, the associated sequences $\langle P^+_j: j\geq 1\rangle$ and $\{Y_j\}_{j\geq 1}$ simplify, and the limit $\nu$ is identified in Subsection \ref{dirichlet_subsection} as a Dirichlet process.
Returning to one of our motivations, we comment that when $\mathcal{X}$ is finite 
these results represent a more constructive view of the limits \eqref{DS_limits} found in \cite{DS}.

\medskip
{\it Organization of the paper.}
We develop notions, make remarks, and state the main results, Theorems \ref{RAM}, \ref{Gem to MCcGEM}, \ref{inhom to MCcGEM}, \ref{occ cor}, \ref{occ cor2}, \ref{selfsimgem}, \ref{momentchain}, and \ref{moments_thm}, in this order, in Section \ref{results_sect}.  Proofs are then given in Section \ref{proofs_sect}.

\section{Statement of results}
\label{results_sect}
We now formalize notation and state our main results, and related remarks about them, in several subsections.  Throughout, we will use the convention that empty sums equal $0$, and empty products are $1$.  Also, $1/0=\infty$, $0/0=0$, and $0^0=1$.  The notation $v^t$ signifies that the vector $v$ is in row form.

\subsection{RAMs, GEMs and MCcGEM laws}
A residual allocation model (RAM) is a way of defining a random probability measure on $\mathbb{N}$ by iteratively assigning a random portion of the unassigned probability remaining to the next integer.

\begin{defi}[Residual Allocation Model - RAM] Let ${\bf X}=\{X_j\}_{j\geq 1}$ be a collection of independent $[0,1]$-valued random variables.  Define
\begin{equation}
\label{RAM_eq1}
P_1 = X_1 \ \ {\rm and \ \ }P_j=X_j\left(1-\sum_{i=1}^{j-1}P_i\right)
 \ \ {\rm for \ \ } j\geq 2.
 \end{equation}
Then, if ${\bf P}=\langle P_j: j\geq 1\rangle$ is a.s. a probability measure on $\N$, that is if $\sum_{j=1}^\infty P_j\stackrel{a.s.}=1$,
we say ${\bf P}$ is a RAM. If ${\bf X}$ consists of iid fractions, and the associated ${\bf P}$ is a RAM, we say
${\bf P}$ is a self-similar RAM.\end{defi}

 Consider now the following identity, verified in Lemma \ref{oneidentity}:  For an arbitrary sequence of numbers $\{a_j\}_{j\geq 1}$ and $k\geq 1$,
\begin{equation}
\label{fraction_identity}
\prod_{j=1}^k (1-a_j) + \sum_{j=1}^k a_j\prod_{i=1}^{j-1} (1-a_i) = 1.
\end{equation}
Then, the sequence in \eqref{RAM_eq1} satisfies $P_j =  X_j\prod_{i=1}^{j-1}(1-X_i)$ for $j\geq 1$ (cf. Proposition \ref{oneidentity_prop}).  Accordingly, we have the useful observation that
${\bf P}$ is a RAM exactly when $\prod_{j\geq 1}(1-X_j)\stackrel{a.s.}=0$.

A specific, well-known example of a RAM is the Griffiths-Engen-McCloskey (GEM) sequence.
 \begin{defi}[GEM] Fix $\theta>0$. Let ${\bf X}=\{X_j\}_{j\geq 1}$ be a sequence of iid variables with common distribution Beta$(1,\theta)$. Then, the self-similar RAM ${\bf P}$, constructed from ${\bf X}$, is said to be a GEM$(\theta)$ distribution.
 
Also, consider a sequence $\{\theta_j\}_{j\geq 1}$ of positive numbers, and let ${\bf X}$ be a sequence of independent random variables where $X_j\sim {\rm Beta}(1,\theta_j)$ for $j\geq 1$.  When the measure ${\bf P}$, found in terms of ${\bf X}$, is a RAM, we will say it is a disordered GEM sequence with parameters $\{\theta_j\}_{j\geq 1}$.
\end{defi}

Now, in a RAM ${\bf P}$, one can clump adjacent probabilities with respect to an increasing sequence ${\bf u}$, marking boundaries of clumps, to form a new probability measure ${\bf P^u}$ on $\mathbb{N}$.  
 
\begin{defi}[Clumped measure]
Let ${\bf u}=\{u_j\}_{j\geq 1}$ be an increasing sequence in $\mathbb{N}\cup\{\infty\}$ with $u_1=1$ and $\lim_{j\rightarrow\infty}u_j=\infty$, and let ${\bf P}$ be a RAM. We clump ${\bf P}$ according to ${\bf u}$ to construct a new probability measure ${\bf P^u}=\langle P^u_j:j\geq 1\rangle$ on $\mathbb{N}$ where, for $j\geq 1$,
\begin{align*}
P^u_j =\left\{\begin{array}{rl}
\sum_{i=u_j}^{u_{j+1}-1}P_i & {\rm if \ } u_j<\infty\\
 0 & {\rm if \ } u_j=\infty.
\end{array}\right.
\end{align*}
\end{defi}

We remark, when ${\bf u}$ takes the value infinity at an entry $u_{j+1}$ in the sequence, necessarily ${\bf P^u}$ is a distribution supported on $\{1,2,\ldots, j\}$.

An immediate question now is when ${\bf P^u}$ is also a RAM. We will show that ${\bf P^u}$ is always a RAM as long as ${\bf u}$ is deterministic.  However, the situation is more involved when a random sequence is used for the clumping.

Specifically, we will be interested in two types of random clumping sequences constructed from a Markov chain ${\bf T} = \{T_i\}_{i\geq 1}$ on the discrete space $\mathcal{X}$. The first sequence ${\bf V}$ comes from considering clumps of repeated values in ${\bf T}$; that is, ${\bf V}$ will keep track of the times when ${\bf T}$ switches values. The second sequence ${\bf W}$ arises in considering the times when ${\bf T}$ returns to its initial value $T_1$.

For example, if ${\bf T}=(1,1,2,2,2,2,4,1,1,5,\ldots)$ is observed, we define ${\bf V} = (1,3,7,8,10, \ldots)$ and ${\bf W}= (1,2,8,9,\ldots)$.  More formally,
Let $V_1=W_1=1$ and, for $j\geq 1$, set
\begin{equation}
\label{V_def}
V_{j+1} =\inf\left\{v>V_j:T_v\neq T_{v-1}\right\} \ \ {\rm and \ \ } 
W_{j+1} =\inf\left\{w>W_j:T_w=T_1\right\}.
\end{equation}
In the case that ${\bf T}$ reaches an absorbing state, denoted $T_\infty$, the chain is eventually constant and ${\bf V}$ is eventually infinite. In the case that $T_1$ is a transient state, the chain returns to the first state finitely many times and ${\bf W}$ eventually takes the value infinity.

Define now ${\bf Y} = \{Y_j\}_{j\geq 1}$ by $Y_j =T_{V_j}$ for $j\geq 1$.  When ${\bf T}$ does not reach an absorbing state, we think of ${\bf Y}$ as the sequence of values taken by ${\bf T}$ without repetition. If however ${\bf T}$ meets an absorbing state $T_\infty$, ${\bf Y}$ will eventually be constant at value $T_\infty$.  

In the following theorem, a reader may like to focus on first pass on the case when ${\bf T}$ possesses no absorbing states and formulas simplify.

In what follows, we will say that a sequence ${\bf z}$ is a `possible' sequence for a Markov chain ${\bf Z}$ on $\mathcal{X}$ if the event $\{Z_i = z_i: 1\leq i\leq n\}$ has positive probability for each $n\geq 1$.

\begin{theorem}[Clumped RAMs]
\label{RAM}
Let ${\bf P}$ be a RAM. Fix an increasing sequence ${\bf u}=\{u_j\}_{j\geq 1}$ in $\mathbb{N}\cup\{\infty\}$ with $u_1=1$ and $\lim_{j\rightarrow\infty}u_j=\infty$. Then,
\begin{itemize}
\item [(1)] ${\bf P^u}$ is a RAM with respect to fractions
${\bf X^u}=\{X^u_j\}_{j\geq 1}$ where
\begin{align*}
X^u_j = \left\{\begin{array}{cl}
\sum_{i=u_j}^{u_{j+1}-1}X_i\prod_{l=u_j}^{i-1}(1-X_l)&\\
\ \ \ \ \ \ \ \ 
 = 1-\prod_{i=u_j}^{u_{j+1}-1}(1-X_i) 
& {\rm if \ } u_j<\infty \\
   1 & {\rm if \ } u_j=\infty.
\end{array}
\right.
\end{align*} 

\end{itemize}

Let now ${\bf T}=\{T_j\}_{j\geq 1}$ be a Markov chain, independent of ${\bf P}$ and with homogeneous transition kernel $Q$. 
\begin{itemize}
\item[(2)] Then, the sequence ${\bf Y} = \{T_{V_j}\}_{j\geq 1}$ is a Markov chain with homogeneous transition kernel $K$ given by
\begin{align*}
K(z,w)  = \left\{\begin{array}{rl}
\frac{Q_{z,w}}{1-Q_{z,z}}  & {\rm for \ } z\neq w;\ Q_{z,z}\neq 1\\
 1  & {\rm for \ } z=w;\ Q_{z,z}=1\\
 0  & {\rm otherwise. \ }
\end{array}\right.
\end{align*}
\end{itemize}

Let ${\bf t}$ be a possible sequence in $\mathcal{X}$ with respect to ${\bf T}$. Let ${\bf y}$ be a possible sequence in $\mathcal{X}$ with respect to ${\bf Y}$.

\begin{itemize}
\item[(3)] Then, ${\bf P^V}\bigr|{\bf T}={\bf t}$ and ${\bf P^W}\bigr|{\bf T}={\bf t}$ are RAMs.
\item[(4)] Also, if ${\bf P}$ is self-similar, ${\bf P^V}\bigr|{\bf Y}={\bf y}$ is a RAM and, when $t_1$ is a recurrent state with respect to ${\bf T}$, ${\bf P^W}\bigr|T_1=t_1$ is a self-similar RAM.
\end{itemize}
\end{theorem}

We remark that the specifications of the fractions and their distributions in items (4) are given in the proof of Theorem \ref{RAM}.   These specifications, in the case when ${\bf P}$ is a GEM$(\theta)$ sequence, are part of Theorem \ref{Gem to MCcGEM}.

Also, in item (4) above, we note that the self-similarity of ${\bf P}$ is important to deduce in full generality that ${\bf P^V}\bigr|{\bf Y}$ is a RAM.  Later, in Example \ref{GEM_example1}, we see that ${\bf P^V}\bigr|{\bf Y}$ may not be a RAM if ${\bf P}$ is not a self-similar RAM.

In addition, we observe that in item (4), when $t_1$ is a transient state, the sequence ${\bf X^W}|T_1=t_1$ eventually takes constant value $1$ since $t_1$ is visited only a finitely many times a.s.  Given $X^W_1|T_1=t_1$ is a nontrivial variable, ${\bf X^W}\bigr|T_1=t_1$ cannot be iid.  However, one may consider an iid sequence $\{Z_i\}_{i\geq 1}$, say on a different probability space, where $Z_1\stackrel{d}{=}X^W_1\bigr|T_1=t_1$, and check that the self-similar RAM formed from fractions $\{Z_i\}_{i\geq 1}$ has the same distribution as ${\bf P^W}\bigr|T_1=t_1$.

We now consider the clumping procedures with respect to a GEM distribution ${\bf P}$.  
It will be convenient to define the notion of a generator kernel or matrix, these terms used interchangeably.

\begin{defi}[Generator kernel]
\label{generator_def}
Let $G=\{G_{i,j}: i,j\in \mathcal{X}\}$ be a square matrix on $\mathcal{X}$.
We say that $G$ is a generator kernel if it satisfies $G_{i,j}\geq0$ for $i\neq j$ and $G_{i,i}=-\sum_{j\neq i}G_{i,j}$.  In addition, we will assume a boundedness condition, $\sup_i |G_{i,i}|<\infty$.\end{defi}

Every matrix of the form $G=\theta(Q-I)$, where $\theta>0$ and $Q$ is a stochastic kernel on $\mathcal{X}$, is a generator matrix.  Moreover, we claim that every generator matrix can be (non-uniquely) decomposed in this fashion:  The final condition in Definition \ref{generator_def} ensures that all entries are bounded, $\sup_{l,k}|G_{l,k}| \leq \sup_i |G_{i,i}|<\infty$, so that a normalizing $\theta$ can be found.  

We also observe that a generator matrix $G$ has a zero row, that is $G_{i,i}=0$ for some $i\geq 1$, exactly when $i$ is an absorbing state for a corresponding $Q$.  In particular, when $G$ does not have zero rows, any corresponding $Q$ does not have absorbing states.

We now formally define the notion of a Markov Chain conditional GEM (MCcGEM) joint distribution on the space $[0,1]^{\N}\times \mathcal{X}^{\N}$, endowed with the product topology and product $\sigma$-field formed in terms of the Borel $\sigma$-fields on $[0,1]$ and $\mathcal{X}$.  This topology is discussed more in Subsection \ref{proof_occ_cor1}.
 By convention, we will say that a Beta$(1,0)$ random variable equals $1$ a.s.
 
\begin{defi}[MCcGEM distribution] With respect to a generator matrix $G$,
let ${\bf Y}$ be a homogeneous Markov chain with initial distribution $\mu$ and transition kernel $K_G$ on $\mathcal{X}$ given by
\begin{align}
K_G(w,z)  = \left\{\begin{array}{rl}
\frac{G_{w,z}}{-G_{w,w}}  & {\rm if }\ \ w\neq z;\ G_{w,w}\neq0
\label{MCcGEM_kernel}\\
 1  & {\rm if }\ \ w=z;\ G_{w,w}=0\\
 0  & {\rm otherwise.}
\end{array}\right.
\end{align}

Consider variables ${\bf X}=\{X_j\}_{j\geq 1}$, on the same probability space as ${\bf Y}$, such that $X_j\bigr|{\bf Y}={\bf y}\sim$ Beta$(1,-G_{y_j,y_j})$ and $\{X_j\bigr|{\bf Y}={\bf y}\}_{j\geq 1}$ are independent.  Define ${\bf P}$ where $P_j=X_j\prod_{i=1}^{j-1}\left(1-X_i\right)$ for $j\geq 1$, and observe that ${\bf P}\bigr |{\bf Y} = {\bf y}$ is a disordered GEM with parameters $\{-G_{y_j,y_j}\}_{j\geq 1}$ (see below).

We say that the pair $({\bf P}, {\bf Y})$ has MCcGEM$(G)$ distribution with respect to $\mu$.  
\end{defi}

To see that ${\bf P}\bigr|{\bf Y}={\bf y}$ is a disordered GEM, we need only observe that ${\bf P}\bigr |{\bf Y}={\bf y}$ is a probability distribution on $\mathbb{N}$.  Here, $\prod_{n\geq 1} (1-X_n)\bigr| \big({\bf Y}={\bf y}\big)=0$ a.s. exactly when $\sum_{n\geq 1} X_n \bigr|{\bf Y}={\bf y}$ diverges a.s.  As the tail $\sigma$-field is trivial, the opposite is the summability $\sum_{n\geq 1}  X_n \bigr|\big({\bf Y}={\bf y}\big)<\infty$ a.s.  By Kolmogorov's $3$-series theorem, and that ${\bf X}|{\bf Y}={\bf y}$ is composed of Beta random variables on $[0,1]$ with means $\{(1-G_{y_j,y_j})^{-1}\}_{j\geq 1}$ and variances dominated by the means, almost sure summability holds exactly
when $\sum_{j\geq 1} |G^{-1}_{y_j,y_j}| <\infty$.  For a generator matrix $G$, this is never the case as the terms $\{|G_{x,x}|\}_{x\in \X}$ are uniformly bounded above.

We now describe a relation between GEM distributions and MCcGEM laws through clumping with respect to a homogeneous Markov chain.

\begin{theorem}[GEM to MCcGEM]
\label{Gem to MCcGEM} 
Let $\theta>0$ and ${\bf P}$ be GEM$(\theta)$ distribution.  Let also ${\bf T}=\{T_j\}_{j\geq 1}$ be an independent homogeneous Markov chain with kernel $Q$ and initial distribution $\mu$.
Recall the associated switch times ${\bf V}$, the clumped distribution ${\bf P^V}$, and the Markov chain ${\bf Y}$ near \eqref{V_def}. 

Then, ${\bf Y}$ is a homogeneous Markov chain with kernel $K_{\theta(Q-I)}$ and ${\bf P^V}|{\bf Y=y}$ is a disordered GEM with parameters $\{\theta(1-Q_{y_j,y_j})\}_{j\geq 1}$, that is
$({\bf P^V},{\bf Y})$ has MCcGEM$(\theta(Q-I))$ distribution with respect to $\mu$. 
\end{theorem}

Some cases of interest are developed in the following examples.

\begin{ex}\label{GEM_example}
\rm
Suppose $\bf{P}\sim$ GEM$(\theta)$ and that $\bf{T}$ is a homogeneous Markov chain with stochastic kernel $Q$ where $Q$ has constant diagonal entries, $Q_{i,i}=q$ for $i\in \X$. By Theorem \ref{Gem to MCcGEM}, ${\bf P^V}\bigr |{\bf Y}$ is a disordered GEM sequence with parameters $\{\theta(1-Q_{y_i,y_i})\}_{i\geq 1}$.  However, since $Q_{y_i,y_i}\equiv q$, we conclude ${\bf P^V}\bigr |{\bf Y}={\bf P^V}$ does not depend on ${\bf Y}$ and is actually a GEM$(\theta(1-q))$ sequence.  In this case, the pair $({\bf P^V}, {\bf Y})$ consists of independent sequences.    

More generally, suppose $\bf{P}$ is any random distribution on $\mathbb{N}$.  Then, indeed, with respect to this Markov chain $\bf{T}$, by the proof of Part (4) of Theorem \ref{RAM} (cf. \eqref{fractions X V}), the fractions ${\bf X^V}$ do not depend on ${\bf Y}$, and so $\bf{P^V}\bigr |{\bf Y}=\bf{P^V}$.
\end{ex}

\begin{ex}\label{GEM_example1}
\rm We now consider a RAM $\bf{P}$ constructed from independent fractions $X_j\sim {\rm Beta}(1/2, 1+j/2)$ for $j\geq 1$. Such a RAM is a member of the well-known 2-parameter GEM$(\alpha,\theta)$ family, here with ${\bf P}\sim$ GEM$(1/2,1)$.  Let $\bf{T}$ be a sequence of iid Bernoulli$(1/2)$ variables.  Thought of as a Markov chain on the $2$-state space $\X=\{1,2\}$, every entry of the stochastic kernel $Q$ of ${\bf T}$ equals $1/2$.  By the discussion in Example \ref{GEM_example}, as the diagonal entries of $Q$ are the constant $q=1/2$, we have ${\bf P^V}\bigr |{\bf Y} = {\bf P^V}$. 

We now observe that ${\bf P^V}$ is not a RAM:  If it were a RAM, consider the associated non-atomic fractions ${\bf X^V}$ (cf. Part (1) of Theorem \ref{RAM}).  Compute
\begin{align*}
\E\left[1-X_1^V|V_2-V_1=m,V_3-V_2=n\right]&=\E\left[\prod_{j=1}^m(1-X_j)\right]\\ & =\prod_{j=1}^m\frac{2+j}{3+j}=\frac3{3+m}\\
\E\left[1-X_2^V|V_2-V_1=m,V_3-V_2=n\right] & =\frac{3+m}{3+m+n}\\
\E\left[(1-X_1^V)(1-X_2^V)|V_2-V_1=m,V_3-V_2=n\right] & =\frac3{3+m+n}.
\end{align*}
Then, 
$\E\left[1-X_1^V\right] =\sum_{m\geq1}\frac3{3+m}(.5)^m$, 
$\E\left[1-X_2^V\right]  =\sum_{n,m\geq1}\frac{3+m}{3+m+n}(.5)^{m+n}$, and
$\E\left[(1-X_1^V)(1-X_2^V)\right]  =\sum_{m,n\geq1}\frac3{3+m+n}(.5)^{m+n}$.  Hence,
$\text{Cov}[1-X_1^V,1-X_2^V]  \approx-.005391$, and so the non-atomic fractions are not independent, and ${\bf P^V}$ cannot be a RAM.
\end{ex}

\subsection{Clumping and time-inhomogeneous Markov chains}
Of course, the notion of clumping can be applied to random probability measures on $\mathbb{N}$, which are not RAMs. 
In particular, to capture the empirical occupation law limit of a Markov chain, we study its local occupations, or clumps of the sequence indexed in time, as it explores the space ${\mathcal X}$.  As noted in the introduction, we will look at these local occupations in reverse order.

Let ${\bf  T}=\{ T_j\}_{j\geq 1}$ be a Markov chain on the discrete space ${\mathcal X}$, without absorbing states. Recall the definition of the switching times ${\bf V}$ (cf. \eqref{V_def}), and let 
$N_n = \min\{i: V_i>n\}$
index the first switch after time $n$.
For $1<k< N_n\leq i$ and $j\geq 1$, define
$$
\tau_{n,1}  =n+1-V_{N_n-1}, \ \  \tau_{n,k} =V_{N_n-(k-1)}-V_{N_n-k}, \ \ {\rm and \ \ }\tau_{n,i}=0.$$
Also, set
\begin{equation}
\label{Y_eqn}
 Y_{n,1} = T_n=T_{V_{N_n-1}}, \ \ Y_{n,k} = T_{V_{N_n-k}}, \ \ {\rm and \ \ } Y_{n,i}  = T_1,
 \end{equation}
 and $P_{n,j}=\tau_{n,j}/n$.
Consider the sequences ${\bf P}_n = \langle P_{n,j}: j\geq 1\rangle$ and ${\bf Y}_n = \{Y_{n,j}\}_{j\geq 1}$.

As a concrete example, consider an observation ${\bf T}=(1,1,1,6,6,1,3,3,3,5,\ldots)$. Then for $n=4$, the local occupations are summarized by eventually constant sequences ${\bf P}_4=(1/4,3/4,0,0,0,0,\ldots)$ and ${\bf Y}_4=(6,1,1,1,1,1,\ldots)$. Similarly, when $n=7$, we have ${\bf P}_7=(1/7,1/7,2/7,3/7,0,0,0,\ldots)$ and ${\bf Y}_7=(3,1,6,1,1,1,1,\ldots)$. For a more general depiction, please refer to the figure in Section \ref{tikpic}.

Hence, for $l\in \mathcal{X}$,
we have generally that
$$\nu_n(l):=\frac1n\sum_{j=1}^n\delta_{T_j}(l) = \sum_{j=1}^\infty P_{n,j}\delta_{Y_{n,j}}(l).$$
In the middle of the display, we see the average Markov chain ${\bf T}$ occupation of state $l$ in the first $n$ steps. On the right-hand side, the sum is over local occupations, or clumps, of state $l$, seen in the chain ${\bf  T}$ through $n$ steps.  The notion suggested by this relation, part of the genesis of this article, is that we may study the limit average occupation law of ${\bf T}$ by investigating the limit of the pair $({\bf P}_n, {\bf Y}_n)$ describing local occupations.

We now focus on a class of time-inhomogeneous Markov chains for which the limits of $({\bf P_n}, {\bf Y_n})$ have succinct representation. Specifically, we consider inhomogeneous Markov chains ${\bf T}$ with transition kernels $\{I+G/n\}$, where $G$ is a generator matrix with no zero entries on the diagonal. A finite space ${\mathcal X}$ case where $G$ was taken to have no zero entries at all was studied in \cite{DS}; see also \cite{Englander}, \cite{Bouguet} for related developments.

In these chains, the clump lengths $V_{k}- V_{k-1}$ are typically growing with $k$, unlike for homogeneous Markov chains.  In particular, rather than an ergodic theorem, it was shown in \cite{DS} (cf. \eqref{DS_limits}) that the occupation laws converge weakly to a nontrivial distribution. Here, we consider a countable space generalization, allowing for reducibility and transient states,  and formulate a characterization of these occupation limits through the reversed clumping device described above.

In the following statement, we say that a matrix is non-negative if all its entries are non-negative. Additionally, weak convergences here are in the sense of finite-dimensional distributions, the natural sense associated to the product space $[0,1]^\N\times \mathcal{X}^\N$ endowed with the product topology.

\begin{theorem}[Time-inhomogenous MC to MCcGEM]
\label{inhom to MCcGEM} Let $G$ be a generator matrix on $\mathcal{X}$ without zero rows. Let $\theta>0$ and $M\in\mathbb{N}$ be such that both $M, \theta >\inf\{r\in\mathbb{R}^+:I+r^{-1}G\text{ is non-negative}\}$, and define $Q=I+G/\theta$.  Let also $\pi$ be a stochastic vector and $\mu$ be a stationary distribution of $Q$ so that 
entry-wise,
\begin{equation}
\label{convergence_condition} 
\pi^tQ^n\rightarrow \mu^t \ {\rm as \ } n\rightarrow\infty.
\end{equation}

 Define kernels $\{K_n\}_{n\geq 1}$ by 
\begin{equation}
K_n=I+\frac Gn\mathbbm{1}(n>M),
\label{K_n_eqn}
\end{equation}
and let ${\bf T}$ be the inhomogeneous Markov chain with transition kernels $\{K_n\}_{n\geq 1}$ and initial distribution $\pi$.  Define $({\bf P}_n,{\bf Y}_n)$ as above with respect to ${\bf T}$, and also define the generator matrix $G'$ by
\begin{equation}
 G'_{ij}=\frac{\mu_j}{\mu_i}G_{ji}\mathbbm{1}(\mu_i\neq 0).
\label{G'_eqn}
\end{equation}

Then, ${\bf Y}_n$ converges weakly to the homogeneous Markov chain ${\bf Y}'$ with kernel $K_{G'}$ and initial distribution $\mu$.  Also, for a possible sequence ${\bf y}$ of ${\bf Y}'$, we have ${\bf P}_n\bigr | {\bf Y}_n ={\bf y}$ converges weakly to a disordered GEM sequence ${\bf P}'$ with parameters $\{-G'_{y_n,y_n}\}_{n\geq 1}$.  Therefore, the associated pairs $({\bf P}_n, {\bf Y}_n)$ converge weakly to $({\bf P}', {\bf Y}')$ with MCcGEM$(G')$ distribution with respect to $\mu$.
\end{theorem}

\begin{ex}
\rm
In the context of Example \ref{GEM_example}, suppose $G$ has constant diagonal entries $g$.  Then, the local occupations of the inhomogeneous Markov chain ${\bf P}_n$ would converge to a GEM$(-g)$ distribution, not just conditionally in terms of a MCcGEM distribution.
\end{ex}

We now characterize the limit occupation law of ${\bf T}$ in a `stick-breaking' form with respect to either a MCcGEM distribution, or a paired GEM distribution and homogeneous Markov chain.  In the following, weak convergence of $\nu_n$ is with respect to the discrete topology on $\Delta_\X$, the space of probability measures on $\X$.

\begin{theorem}[Occupation laws to MCcGEM and stick-breaking measures]
\label{occ cor} 
Consider the setting and assumptions of Theorem \ref{inhom to MCcGEM}.  Observe that $\mu$ is a stationary distribution of $Q'= I+ G'/\theta$, and let ${\bf T}'$ be the homogeneous and stationary Markov chain with kernel $Q'$ and initial distribution $\mu$.  Let ${\bf P}^+$ be a GEM$(\theta)$ sequence independent of ${\bf T}'$.

Then, $\nu_n = \left\langle\frac1n\sum_{j=1}^n\delta_{T_j}(l): l\in \mathcal{X}\right\rangle\xrightarrow{d} \nu$, where
\begin{equation}
\label{occ_cor_eqn}
\nu\stackrel d=\left\langle\sum_{j=1}^\infty P'_j\delta_{Y'_j}(l): l\in \mathcal{X}\right\rangle \stackrel d=\left\langle\sum_{j=1}^\infty P^+_j\delta_{T'_j}(l): l\in \mathcal{X}\right\rangle .
\end{equation}
\end{theorem}

In a sense, reversing the procedure, starting from the stick-breaking process $\sum_{j\geq 1} P^+_j \delta_{T'_j}$, we may identify it as the limit of the occupation measure of a matched time-inhomogeneous Markov chain, almost a corollary of Theorem \ref{occ cor}.

\begin{theorem}[Stick-breaking measures to Occupation laws]\label{occ cor2} 
Let $\theta>0$ and ${\bf P}^+$ is a GEM$(\theta)$ sequence.  Let also $\tilde Q$ be a stochastic matrix without absorbing states and with stationary distribution $\mu$.
Suppose ${\bf T}'$ is an independent homogeneous Markov chain with kernel $\tilde Q$ starting from $\mu$.  
 
 Then,
\begin{equation*}
\left\langle \sum_{j=1}^\infty P^+_j \delta_{T'_j}(l): l\in \X\right\rangle \stackrel d = \nu,
\end{equation*}
where $\nu\stackrel d = \lim_{n\rightarrow\infty}\nu_n$ is the occupation law defined with respect to an inhomogeneous Markov chain ${\bf T}$, as in the setting of Theorem \ref{inhom to MCcGEM}, with respect to generator matrix $\tilde G'$, starting from any distribution $\pi$ satisfying $\pi^t(\tilde Q')^n\rightarrow\mu^t$ entry-wise.  Here, $\tilde G'$ and $\tilde Q'$ are given by
 $\tilde G_{ij}' = \big(\mu_j/\mu_i\big)\tilde G_{j,i}{\mathbbm 1}(\mu_i\neq 0)$ where $\tilde G = \theta(\tilde Q-I)$, and $\tilde Q' = I + \tilde G'/\theta$. \end{theorem}

In the next two subsections, we discuss remarks on Theorems \ref{inhom to MCcGEM} and \ref{occ cor},
and a case when the random measure $\nu$ is a Dirichlet process.
 
\subsubsection{Remarks}
\label{remarks subsection}
We now make several comments on Theorems \ref{inhom to MCcGEM} and \ref{occ cor}.
\vskip .1cm

{\bf 1.} Although we have specified that $G$ has no zero rows in Theorems \ref{inhom to MCcGEM} and \ref{occ cor}, and therefore no absorbing states for ${\bf T}$, one can extend some of the statements trivially to the case when there are absorbing states.  In particular, when the limit $\mu$ is the unit point mass at an absorbing state $z$ of $Q$, we have $G'_{z,z}=G_{z,z}=0$ and $K_n(z,z)=K_{G'}(z,z)=1$.  Then, the state $z$ is also an absorbing state for the inhomogeneous Markov chain ${\bf T}$, reached in finite time a.s. starting from $\pi$.  Also, the chain ${\bf T}'$, starting from $\mu$, is the constant sequence of $z$'s.  In addition, the limit of $Y_{n,1}$ is $z$, and $P_{n,1}$ tends to $1$ a.s.    
We conclude that ${\bf P}_n$ converges weakly to ${\bf P}'=\langle 1, 0,\ldots\rangle$, a GEM with constant fractions $1={\rm Beta}(1, 0)$.  Moreover, the empirical distribution $\nu_n$ of the chain ${\bf T}$ converges weakly to $\delta_z$.  We also observe that $\sum_{j\geq 1} P'_j\delta_{Y_j'}$, and also $\sum_{j\geq 1}P_j^+\delta_{T'_j}$ both equal $\delta_z$ in distribution.

\vskip .1cm
{\bf 2.} 
There is a degree of freedom in picking a pair $(\theta,Q)$. However, when specifying a MCcGEM distribution, each valid pair corresponds to the same generator matrix $G$ in this context. On the other hand, this family of pairs $({\bf P}^+,{\bf T}')$ of a GEM distribution and Markov chain, indexed in $\theta$, will have different joint distributions, although they all correspond to a single measure $\nu$.  We explore this notion in the case of Dirichlet processes in Subsection \ref{dirichlet_subsection} below.

\vskip .1cm

{\bf 3.} The convergence \eqref{convergence_condition} is a condition on the structure of positive recurrent states of the homogeneous Markov chain ${\bf T^Q}$ run with kernel $Q=I+G/\theta$.  Since the limit $\mu$ is a stationary distribution with respect to $Q$, the chain must have a positive recurrent state, and $\mu$ is positive only on such states.   The initial distribution $\pi$ must be such that observation of a positive recurrent state occurs with probability 1.

 In general, $\mu$ depends on $\pi$ when there is more than one irreducible class of positive recurrent states.  We note, along with positive recurrent states, there may also be null recurrent and transient states associated with $Q$.

In the case that $Q$ has a single class of positive recurrent states, then $\mu$ will be the unique stationary distribution associated with $Q$ and will not depend on $\pi$.
 
 It could be that $Q$ has an infinite number of null recurrent or transient states, in addition to positive recurrent states.   But, the requirement that $\mu$ be stochastic means that the chain ${\bf T^Q}$ cannot visit a null recurrent state or remain indefinitely on transient states a.s.  This reflects that the limit of $({\bf P_n}, {\bf Y_n})$ corresponds to the long time average occupations of states in ${\mathcal X}$.

\vskip .1cm

{\bf 4.}
Any null recurrent or transient state of the chain run with $Q$ corresponds to a zero row of $G'$ or in other words an absorbing state for the chains ${\bf T}'$ and ${\bf Y}'$.
However, such absorbing states are never visited by ${\bf T}'$:
The initial distribution $\mu$ is a stationary distribution of $Q$, which vanishes on these states.  Moreover, as $\mu$ is also a stationary distribution of $Q'$, the chain ${\bf T}'$ can only move on the positive recurrent states of ${\bf T^Q}$, the states $\{i\in \mathcal{X}: \mu_i>0\}$.

Similarly, starting from $\mu$, the chain ${\bf Y}'$ moves only on states $\{i\in \mathcal{X}: \mu_i>0\}$, given that $G'_{w,z}=K_{G'}(w,z)=0$ when either $\mu_z=0$ or $\mu_w=0$ and $w\neq z$.
  
Also, we comment that the chain ${\bf T}'$ run with $Q'$ is a form of time-reversal of ${\bf T^Q}$ with respect to stationary distribution $\mu$, reflecting the reverse chronological construction of the ${\bf Y}_n$ sequences.

\subsubsection{Dirichlet process limits}
\label{dirichlet_subsection}
In a particular case of Theorem \ref{occ cor}, we observe that we may recover Dirichlet processes.  Suppose $\mu(i)>0$ for all $i\in \X$.
When 
$Q$ has constant rows equal to $\mu^t$, the Markov chain ${\bf T'}$ has transition kernel $Q'=Q$, and therefore ${\bf T}'$ is an iid sequence with common distribution $\mu$. Then, $\nu= \sum_{j\geq 1} P^+_j \delta_{T'_j}$, formed from a GEM$(\theta)$ sequence ${\bf P}^+$ and an independent sequence of iid random variables ${\bf T}'$, is the `stick-breaking' representation of a Dirichlet process with parameters $\theta$ and measure $\mu$ on the discrete space $\mathcal{X}$ (cf. \cite{JSethuraman}). Specifically, as noted in the introduction, when $\mathcal{X}$ is finite we have that $\nu$ is a Dirichlet distribution with parameters $\{\theta \mu_j\}_{j\in \mathcal{X}}$. (cf. \cite{Donnelly_Tavare}, \cite{Hirth}).

Moreover, since the distribution of $\nu$ is determined by $G$, there is a degree of freedom in specifying $G$ via a pair $(\theta,Q)$.  Write $G$ in two forms: (1) $G=\theta(Q-I)$ where $\theta>0$ and $Q$ is stochastic with constant rows $\mu^t$, and also (2) $G=\tilde\theta(\tilde Q-I)$ where $\tilde\theta>0$, $\theta\neq\tilde\theta$, and $\tilde Q$ is stochastic. Then again, $\tilde Q=\tilde Q'$ and via Theorem \ref{occ cor}, we recover a different stick-breaking representation, 
$\sum_{j=1}^\infty P_j^{\tilde \theta}\delta_{T^{\tilde Q}_j}$, of the Dirichlet process with parameters $\theta$ and $\mu$, in terms of GEM$(\tilde\theta)$ sequence $\mathbf P^{\tilde \theta}$ and an independent homogeneous Markov chain $\mathbf{T^{\tilde Q}}$ with $T^{\tilde Q}_1\sim\mu$ and kernel $\tilde Q$.  

Here, $\tilde Q=\frac\theta{\tilde\theta}Q+(1-\frac\theta{\tilde\theta})I$ is the weighted average of $Q$ and $I$.  Since $\tilde Q$ no longer has constant rows, $\mathbf{T^{\tilde Q}}$ no longer consists of iid variables. The chain ${\bf T^{\tilde Q}}$ is, in a sense, a more or less `sticky' version of an iid $\sim\mu$ sequence depending on the weight of $I$ in the weighted average relation for $\tilde Q$.

\subsection{Self-similarity of the occupation laws}
At this point, it is natural to ask for other ways to understand the laws in Theorem \ref{occ cor}.   Consider the general random measure
\begin{equation}
\label{nu_def}
\nu \stackrel d= \left\langle\sum_{j=1}^\infty P_j\delta_{T_j}(l): l\in \mathcal{X}\right\rangle,
\end{equation}
where ${\bf P}$ is a self-similar RAM composed of 
fractions ${\bf X}$,
and ${\bf T}$ is an independent homogeneous Markov chain with transition kernel $Q$ and initial distribution $\mu$, assigning zero probability to any transient state of $Q$.  
We remark that $\nu$ reduces to the measure in Theorem \ref{occ cor} when ${\bf P}\sim$ GEM$(\theta)$ and $\mu$ is a stationary vector of $Q$. 
We first discuss an example.

\begin{ex}\rm
As we have noted earlier, if ${\bf P}\sim$ GEM$(\theta)$ and ${\bf T}$ is an independent sequence of iid variables with distribution $\mu$, the measure $\nu$ is the `stick-breaking' representation of the Dirichlet process with parameters $\theta$ and measure $\mu$ on $\mathcal{X}$. Following \cite{JSethuraman}, a self-similarity relation can be deduced:
$$\nu \stackrel d= X_1\delta_{T_1} + (1-X_1)\tilde \nu,$$
where $\tilde\nu\stackrel{d}{=}\nu$ is another random measure, and $X_1\sim {\rm Beta}(1,\theta)$, $T_1\sim \mu$ and $\tilde \nu$ are independent.  From such an equation, the Dirichlet process characterization of $\nu$ with parameters $\theta$ and measure $\mu$ on $\mathcal{X}$ follows from classical considerations.  Moreover, this relation is central in calculation of a posterior distribution, given say $X_1$, when $\nu$ is thought of as a law on priors.  See also the recent work \cite{Last} and \cite{SL} on related integral characterizations.  
\end{ex}

We now define a more general notion of self-similarity.  This notion is well known (cf. \cite{Goldie} among other references).  With respect to a measurable space $(\mathcal{A}, \mathcal{B}_{\mathcal{A}})$, let $\mathbb{P}_{\mathcal{A}}$ be the space of probability measures on $(\mathcal{A}, \mathcal{B}_{\mathcal{A}})$.  Let $\mathbb{F}_{\mathcal{A}}$ be the smallest $\sigma$-field generated by sets of the form $\Big\{\{\chi: \chi(A)<r\}: A\in \mathcal{B}_{\mathcal{A}}, r\in [0,1]\Big\}$.

\begin{defi}[Self-similar random measure]\label{self} We say that the law of a random distribution $\chi$ on $(\mathbb{P}_{\mathcal{A}}, \mathbb{F}_{\mathcal{A}})$ is self-similar with respect to $(\eta,X)$ if it satisfies
\begin{equation}
\label{self_eqn}
 \chi(\cdot)\stackrel d=X\eta(\cdot )+(1-X) \tilde\chi( \cdot ),
 \end{equation}
where $X$ is a $[0,1]$-valued random variable, $\eta$ is a random distribution on $\mathbb{P}_{\mathcal{A}}$, and $\tilde \chi$ is random measure with the same distribution as $\chi$ and independent of $(\eta,X)$, defined on the space $[0,1]\times \mathbb{P}_{\mathcal{A}}\times \mathbb{P}_{\mathcal{A}}$.
\end{defi}

The key is that such self-similarity may uniquely identify a distribution.  The following is part of Lemma 3.3 in \cite{JSethuraman}; see also \cite{Goldie} for more involved statements.  For the convenience of the reader, a proof is given in Subsection \ref{self_similar}.

\begin{lemma}\label{unique}
There exists a unique in law self-similar random measure $\chi$ on $(\mathbb{P}_{\mathcal{A}}, \mathbb{F}_{\mathcal{A}})$ with respect to $(\eta,X)$ when $\mathcal P(X=0)<1$.
\end{lemma}

\medskip
We now state that $\nu$ defined in \eqref{nu_def} is self-similar in a certain way. Let ${\bf X}=\{X_j\}_{j\geq 1}$ be the iid fractions from which ${\bf P}$ is constructed.  For each recurrent state $i$ of $Q$, let ${\bf T}^{i}$ be a Markov chain with transition kernel $Q$ and initial value $T^{i}_1=i$, independent of ${\bf X}$ and $(\nu,T_1)$. Define the finite cycle length and associated clumped residual fraction,
$$W^i:=\inf\{j>1:T_j^i=i\} \ \ {\rm and \ \ } X^i:=\sum_{j=1}^{W^i-1}X_j\prod_{l=1}^{j-1}(1-X_l).$$

Set
$$\eta^i:=\left(X^i\right)^{-1}\sum_{j=1}^{W^i-1}\left[X_j\prod_{l=1}^{j-1}(1-X_l)\right]\delta_{T_j^i} \ \ {\rm and \ \ } \nu^i := \nu\bigr | T_1=i.$$

\begin{theorem}[Type of self-similarity]
\label{selfsimgem}
The law of $(\nu,T_1)$ uniquely satisfies the following:   
Marginally, $T_1\sim \mu$ and, for each 
recurrent 
state $i$ of $Q$,
\begin{equation}
\label{self_similarity equation}
\nu^i \stackrel d= X^i\eta^i + (1-X^i)\tilde \nu^i,
\end{equation}
where $\tilde \nu^i$ is random measure with the same law as $\nu^i$, such that $\tilde \nu^i$ and $(\eta^i,X^i)$ are independent.
\end{theorem}

If $\nu$ is thought of as a distribution on priors, the notion of a posterior distribution given a cycle of data $X^i$ might be considered from the self-similarity \eqref{self_similarity equation}. However, we remark that such a computation does not seem as tractable as in the case $\nu$ is a Dirichlet process (cf. \cite{JSethuraman}).

One might ask what happens when starting from a transient state $T_1 = i$.  In this case, there is positive chance that one will not return to $T_1$.  As above, one may write down a first `cycle' decomposition 
but, because $W^i$ may not be finite, the decomposition does not immediately lead to a `self-similarity' equation as in \eqref{self_similarity equation}. 
However, 
one might consider a stick-breaking construction, on a different probability space, which does lead to a `self-similarity' equation.  Indeed, following the discussion after Theorem \ref{RAM}, consider for transient states $T_1=i$ an iid sequence of pairs
$\{(Z_j,\eta_j)\}_{j\geq 1}$ with common distribution $(X^i,\eta^i)\bigr | T_1=i$, and form a stick-breaking construction, which after an exercise is seen to be equivalent-in-distribution to $\nu^i$: 
$$\nu^i\stackrel d=\left\langle\sum_{j=1}^\infty \eta_j(l)Z_j\prod_{r=1}^{j-1}(1-Z_r):l\in\mathcal{X}\right\rangle.$$
Then, $\nu^i \stackrel d= Z_1\eta_1 + (1-Z_1)\tilde\nu^i$ where $\tilde\nu^i\stackrel d= \nu^i$, and $\tilde\nu^i$ and $(\eta_1, Z_1)$ are independent.

\subsection{Moments of the occupation laws}

We first recall Theorems 1.3 and 1.4 in \cite{DS}: Suppose $G$ is a generator matrix on finite state space $\mathcal{X}=\{1,2,...,k\}$ with no 0 entries. By identification through its moments, the limiting occupation random variable $\nu=\langle\nu_1,...,\nu_k\rangle$ (cf. \eqref{occ_cor_eqn}) of an inhomogeneous Markov chain $\mathbf{T}$ with kernels of the form $K_n=I+\frac{G}{n}1(n>M)$ was found:  Let $\N_0 = \{0,1,2,\ldots\}$ be the whole numbers. For ${\bf m}=(m_1,\ldots, m_k)\in\N_0^k$ and $N=\sum_{j=1}^k m_j>0$, we have
\begin{equation}
\label{DS_moments}
\E\left(\prod_{j=1}^k \nu_j^{m_j}\right)={N\choose m_1,...,m_k}^{-1}\sum_{\sigma\in\mathbb{S}({\bf m})}\mu_{\sigma_1}\prod_{j=1}^{N-1}\left[\left(I-\frac Gj\right)^{-1}\right]_{\sigma_j,\sigma_{j+1}}
\end{equation}
where $\mu$ is the unique stochastic eigenvector of $G$, and $\mathbb{S}({\bf m})$ is the set of ${N\choose m_1,...,m_k}$ distinct permutations of the list of $N$ integers consisting of $m_1$ many 1's, $m_2$ many 2's, up through $m_k$ many $k$'s.

In particular, when $G$ can be written $\theta(Q-I)$ where $\theta>0$ and $Q$ is the stochastic matrix with constant rows $\mu$, the expectation reduces to the moments of the Dirichlet distribution with parameters $\theta\mu$: $\E\left[\prod_{j=1}^k\nu_j^{m_j}\right]=((\theta)_N)^{-1}\prod_{j=1}^k(\theta\mu_j)_{m_j}$ where 
$(a)_n={\Gamma(a+n)}/{\Gamma(a)}$ 
is the Pochhammer symbol, that is a rising factorial.  However, when $G$ is not of this form and $k\geq 3$, one can see that the moments may not describe a Dirichlet distribution.

In this context, 
we detail now some more descriptions of these laws.  Observe that the matrix $\big(I-G/j\big)^{-1}$ for $j\geq 1$ is a resolvent operator with respect to the transition function $\{\P^G_s: s\geq 0\}$ of a continuous time Markov chain with generator $G$.  In particular, it is standard to write
$$\left(I-\frac{G}{j}\right)^{-1}_{l,m} = \int_0^\infty je^{-js}\P^G_s(l,m)ds.$$
As a consequence, $\widetilde K_j:=\big(I-G/j\big)^{-1}$ itself is a stochastic kernel on $\X$.  

In the Dirichlet case $G=\theta(Q-I)$, where $\theta>0$ and each row of $Q$ is the stationary measure $\mu$, one can see by calculating via the backward equation 
$\frac{d}{ds}\P^G_s = G\P^G_s$ and $\P^G_0 = I$ that 
$$\widetilde K_j=I+\frac{G}{j+\theta}.$$

 More generally, let ${\bf Z}=\{Z_k\}_{k\geq 1}$ be the inhomogeneous Markov chain with initial distribution $\mu$ and transition kernels $\widetilde K_n$ for $n\geq 1$.

We first observe a type of `duality' relation between the moments of $\nu$ and ${\bf Z}$.

\begin{theorem}[Recasting moments I]
\label{momentchain}
Recall the setting of \cite{DS} given above.  Then $\widetilde K_n = K_n + O(n^{-2})$ and the measure $\nu$ with respect to ${\bf T}$ is also the occupation law with respect to ${\bf Z}$,
\begin{align}
\nu & \stackrel d=\lim_{n\rightarrow\infty}\frac1n\left\langle\sum_{j=1}^n \delta_{Z_j}(i)\right\rangle_{i\in\mathcal{X}}.
\label{newocc}
\end{align}
Moreover, the moments may be expressed in terms of ${\bf Z}$,
\begin{align}
\label{newmoments}
\E\left[\prod_{i=1}^k\nu_i^{m_i}\right] & ={N\choose m_1,...,m_k}^{-1} \P\left(\sum_{j=1}^N1_i(Z_j)=m_i:1\leq i\leq k\right),
\end{align}
and, in particular, $\E\left[\nu_i^N\right]=P(Z_1=...=Z_N=i)$.
\end{theorem}

Alternatively, we now recast the moment result \eqref{DS_moments} in an algebraic form where it can be more easily exploited.   Let $p_{min}(\lambda)$ be the minimal polynomial of $G$ and $q(\lambda)=\sum_{k=0}^na_k\lambda^k$ be the polynomial such that $p_{min}(\lambda)=\lambda q(\lambda)$. Define, for $j\geq 0$,
\begin{equation}
\label{q equation}
p_j(\lambda):= \frac{p_{min}(\lambda)-p_{min}(j)}{\lambda-j} =\sum_{k=0}^n\lambda^k\sum_{l=k}^n a_l j^{l-k}.
\end{equation}

\begin{theorem}[Recasting moments II]
\label{moments_thm}
We have
$p_0(G)/q(0)$ is the matrix with constant rows $\mu$, and $p_j(G)/q(j) = \widetilde K_j$ for $j\geq 1$.  As a consequence,
 for ${\bf m}\in\mathbb{N}_0^k$ with $\sum _{i=1}^km_i=N>0$ and fixed constant $\sigma_0=1$,
\begin{equation}
\label{alg_mom}
\E\left[\prod_{i=1}^k\nu_i^{m_i}\right]={N\choose m_1,...,m_k}^{-1}\sum_{\sigma\in\mathbb{S}({\bf m})}\prod_{j=0}^{N-1}\left[\frac{p_j(G)}{q(j)}\right]_{\sigma_j\sigma_{j+1}}.
\end{equation}
\end{theorem}

One can now recover the moments of the marginals.

\begin{cor}[Marginals] \label{marginalmom}
Let $\{\lambda_l\}_{l=1}^n$ be the non-zero roots of $q$, all of which are non-zero eigenvalues of $G$.  Let also $\{\gamma_{i,l}\}_{l=1}^n$ be the zeros of $[p_j(G)]_{ii}$ considered as a function of $j$.  Then,
\begin{align}
\E[\nu_i^N]=\prod_{l=1}^n\frac{\left(-\gamma_{i,l}\right)_N}{\left(-\lambda_l\right)_N}\label{margform}
\end{align}
\end{cor} 
 
Interestingly, when $\Lambda=\{\lambda_l\}_{l=1}^n$ and $\Gamma_i=\{\gamma_{i,l}\}_{l=1}^n$ are real and pairwise ordered $\lambda_j<\gamma_{i,j}<0$, we recognize these marginal moments as the product of the $N$th order moments of independent Beta$(-\gamma_{i,l},\gamma_{i,l}-\lambda_l)$ variables for $1\leq l\leq n$. 

In the Dirichlet case, when $G$ is of the form $\theta(Q-I)$ where $\theta>0$ and $Q$ is stochastic with constant rows $\mu$, we have $q(j)=j+\theta$ and $p_j(\lambda)=\lambda+j+\theta$. This corresponds to $[p_j(G)]_{ii}=j+\theta\mu_i$ and $\E[\nu_i^N]=\frac{(\theta\mu_i)_N}{(\theta)_N}$, the $N$th order moments of a Beta$(\theta\mu_i,\theta(1-\mu_i))$ variable or equivalently the $i$th marginal of a Dirichlet variable with parameters $\theta\mu$.

However, in general, $\Lambda$ and $\Gamma_i$ need not be sets of real numbers, and \eqref{margform} gives the moments of beta products in a sense with complex parameters.

The marginal density function $f_i$ of $\nu_i$ can be written in terms of Meijer G-functions, typically denoted $G^{M,N}_{P,Q}\left(\left.\begin{array}{c}\vec{a}\\ \vec{b}\end{array}\right|z\right)$ where $M\leq Q$ and $N\leq P$ are non-negative integers, $\vec{a}\in\mathbb{C}^P$, and $\vec{b}\in\mathbb{C}^Q$. Given $\Lambda$ and $\Gamma_i$, $f_i$ is given by
\begin{align*}
f_i(x) & =\left[\prod_{l=1}^n\frac{\Gamma(-\lambda_l)}{\Gamma(-\gamma_{i,l})}\right]x^{-1}G^{n,0}_{n,n}\left(\left.\begin{array}{c}-\Lambda\\-\Gamma_i\end{array}\right|x\right)\mathbbm{1}(0<x<1).
\end{align*}

The class of Meijer-G functions includes generalized hypergeometric functions, among others. For a thorough review of Meijer G-functions, their specification, and connection to Beta products via the Mellin transform, see \cite{MathaiSaxena}. See also \cite{Dufresne} for a discussion of the distributional properties of the product of two Beta variables with complex parameters with an application to risk theory.

\section{Proofs}
\label{proofs_sect}

We first note a standard algebraic identity which leads to useful formulas for RAMs.  Recall our conventions specified at the beginning of section \ref{results_sect}.

\begin{lemma}\label{oneidentity} For any sequence of numbers $a_j$ and integer $k\geq 1$, we have
\begin{equation}
\label{one_eqn}
\prod_{j=1}^k(1-a_j)+\sum_{j=1}^k a_j\prod_{i=1}^{j-1}(1-a_i)=1.
\end{equation}
\end{lemma}

\noindent {\it Proof.}
We proceed by an induction. Equation \eqref{one_eqn} is trivially true for $k=1$: $(1-a_1)+a_1=1$. If it is true for $k-1$, then
the left-hand side of \eqref{one_eqn} equals
\begin{align*}
& \prod_{j=1}^{k-1}(1-a_j)-a_k\prod_{j=1}^{k-1}(1-a_j)+\sum_{j=1}^{k-1}a_j\prod_{i=1}^{j-1}(1-a_i)+a_k\prod_{j=1}^{k-1}(1-a_j)\\
= & \prod_{j=1}^{k-1}(1-a_j)+\sum_{j=1}^{k-1}a_j\prod_{i=1}^{j-1}(1-a_i)=1. \quad \quad \quad \quad \quad \quad \quad \quad \quad \quad \quad \quad \quad \quad \quad \quad \quad \quad \square
\end{align*}

\begin{prop}
\label{oneidentity_prop} 
Consider a distribution ${\bf P}=\langle P_j: j\geq1 \rangle$ on $\N$ and factors ${\bf X}=\{X_j\}_{j\geq 1}$ with 
$$X_j = \left\{\begin{array}{cl}
P_j\left(1-\sum_{i=1}^{j-1} P_i\right)^{-1}& \ {\rm  when \ }\sum_{i=1}^{j-1} P_i<1\\
1 & \ {\rm  otherwise.}
\end{array}\right.
$$  Then,  
$P_j = X_j\prod_{i=1}^{j-1}(1-X_i)$ for $j\geq 1$.

  In particular, if ${\bf P}$ is a RAM constructed from ${\bf X}=\{X_j\}_{j\geq 1}$, for $1\leq k\leq r$, we have
\begin{equation}
\label{P_eqn}
\sum_{j=1}^r P_j = 1- \prod_{j=1}^r (1-X_j), \ \ {\rm and \ \ } \sum_{j=k}^r P_j =\prod_{j=1}^{k-1}(1-X_j)\left[ 1-\prod_{j=k}^r(1-X_j)\right].
\end{equation}
\end{prop}

\begin{proof}
Part (I) follows from \eqref{one_eqn} by an induction:  Trivially, $P_1 = X_1$.  Suppose $P_{k} = X_{k}\prod_{i=1}^{k-1}(1-X_i)$ for $k\leq j$ and so, by \eqref{one_eqn}, we have $\prod_{k=1}^j(1-X_k) = 1- \sum_{k=1}^jP_k$.  Then, $P_{j+1} = X_{j+1}\left(1-\sum_{k=1}^{j}P_k\right) = X_{j+1}\prod_{k=1}^j (1-X_k)$.

For Part (II), the lines in \eqref{P_eqn} follow from Part (I) and \eqref{one_eqn}.
\end{proof}

\subsection{Proof of Theorem \ref{RAM}: Clumped RAMs}

Let ${\bf P}$ be a RAM, and let ${\bf X}=\{X_j\}_{j\geq 1}$ be the independent proportions from which ${\bf P}$ is constructed.  From Proposition \ref{oneidentity_prop}, for $j\geq 1$, we have $P_j=X_j\prod_{i=1}^{j-1}(1-X_i)$.

Let ${\bf u}=\{u_j\}_{j\geq 1}$ be an increasing sequence in $\N\cup\{\infty\}$ with $u_1=1$ and $\lim_{j\rightarrow\infty}u_j=\infty$.
Define new proportions ${\bf X^u} = \{X^u_j\}_{j\geq 1}$ from ${\bf X}$, using Proposition \ref{oneidentity_prop} again:  For $j\geq 1$,
\begin{align}\label{Ram_2}
X^u_j  =\left\{\begin{array}{cl}
\sum_{i=u_j}^{u_{j+1}-1}X_i\prod_{l=u_j}^{i-1}(1-X_l) & \\
\ \ \ \ \ \ \ =1-\prod_{i=u_j}^{u_{j+1}-1}(1-X_i) & {\rm if \ } u_j<\infty
\\
1 & {\rm if \ } u_j=\infty.\end{array}\right.
\end{align}
Recall, for $j\geq 1$, that $P^u_j = \sum_{i=u_j}^{u_{j+1}-1} P_i$ when $u_j<\infty$ and $P^u_j = 0$ otherwise, and ${\bf P^u}= \{P^u_j\}_{j\geq 1}$.

We now proceed to the proofs of Parts (1)-(4).

\subsubsection{Proof of Part (1)}  We now verify that ${\bf P^u}$ is a RAM with respect to fractions ${\bf X^u}$: Let $B=\sup\{j:u_j<\infty\}$.  For $1\leq j\leq B$, noting 
\eqref{Ram_2}, write
\begin{eqnarray*}
P_j^u  &=&\sum_{i=u_j}^{u_{j+1}-1}P_i \ = \ 
\sum_{i=u_j}^{u_{j+1}-1}X_i\prod_{l=1}^{i-1}(1-X_l)\\
&=& \left[\sum_{i=u_j}^{u_{j+1}-1}X_i\prod_{l=u_j}^{i-1}(1-X_l)\right] \prod_{l=1}^{u_j-1}(1-X_l)\ = \ 
 X_j^u\prod_{l=1}^{u_j-1}(1-X_l)\\
&=&  X^u_j \left[\prod_{l=u_1}^{u_2-1}(1-X_l)\right]\cdots \left[\prod_{l=u_{j-1}}^{u_j-1}(1-X_l)\right]\ = \ 
 X_j^u\prod_{i=1}^{j-1}(1-X^u_i).
\end{eqnarray*}

For $j>B$, note $P_j^u=0$ and
$\prod_{i=1}^B(1-X_i^u)=1-\sum_{i=1}^B P_i^u=1-\sum_{i=1}^\infty P_i=0$.
Then, $X_j^u\prod_{i=1}^{j-1}(1-X_i^u)=0=P_j^u$.

Since ${\bf X}=\{X_j\}_{j\geq 1}$ is composed of independent variables, so is ${\bf X^u}=\{X^u_j\}_{j\geq 1}$.  Hence, as $\sum_{j\geq 1}P^u_j=\sum_{j\geq 1} P_j\stackrel{a.s.}{=} 1$, by definition, ${\bf P^u}$ is a RAM constructed from independent proportions ${\bf X^u}$. \qed

\subsubsection{Proof of Part (2)}

Let ${\bf y}=\{y_i\}_{i\geq 1}$ be a possible sequence for ${\bf Y}$ in $\mathcal{X}$.  Define $J=\inf\{j\geq 1: y_j=y_{j+1}\}$.  Then, ${\bf y}$ is then either non-repeating and $J=\infty$, or ${\bf y}$ is non-repeating until reaching a finite time $J$, after which the sequence is constant.

For $1\leq n<J$, the event that $Y_i=y_i$ for $1\leq i\leq n$ means the chain ${\bf T}$ starts in $y_1$, staying there until time $V_2$, when it switches to $y_2$, remaining there until time $V_3$, and so on up to time $V_n$ when it moves into $y_n$.
Write for $n<J$ that
\begin{eqnarray}
 && \P\left(Y_i=y_i: 1\leq i\leq n\right)\nonumber\\
&&\ =  \ \sum_{l_1=1}^\infty \cdots \sum_{l_{n-1}=1}^\infty \P\left ( Y_i=y_i: 1\leq i\leq n, \ \ {\rm and \ \ } V_{i+1}-V_i=l_i: 1\leq i\leq n-1\right)\nonumber\\
&&\ = \  \sum_{l_1 = 1}^\infty \cdots \sum_{l_{n-1}=1}^\infty \P(T_1=y_1)\prod_{i=1}^{n-1}Q_{y_iy_i}^{l_i-1} Q_{y_iy_{i+1}}\nonumber \\
&&\ = \  \P(T_1=y_1)\prod_{i=1}^{n-1}\frac{Q_{y_iy_{i+1}}}{1-Q_{y_iy_i}}\ = \ 
    \P(T_1=y_1)\prod_{i=1}^{n-1}K(y_i,y_{i+1}).
\label{2_seq}
\end{eqnarray}

Suppose $J<\infty$.  Then, $y_J$ is an absorbing state of ${\bf T}$ and, for $i\geq J$, we have $Q_{y_i,y_i}=K(y_i,y_i)=K(y_i,y_{i+1})=1$. Define $l_J=\infty$ and write for $n\geq J$ that
\begin{eqnarray*}
 && \P\left(Y_i=y_i: 1\leq i\leq n\right)\\
&&\ =  \ \sum_{l_1=1}^\infty \cdots \sum_{l_{J-1}=1}^\infty \P\left ( Y_i=y_i: 1\leq i\leq J, \ \ {\rm and \ \ } V_{i+1}-V_i=l_i: 1\leq i\leq J\right)\\
&&\ = \  \sum_{l_1 = 1}^\infty \cdots \sum_{l_{J-1}=1}^\infty \P(T_1=y_1)\prod_{i=1}^{J-1}Q_{y_iy_i}^{l_i-1} Q_{y_iy_{i+1}}
\ = \ 
\P(T_1=y_1)\prod_{i=1}^{J-1}\frac{Q_{y_iy_{i+1}}}{1-Q_{y_iy_i}}\\
&&\ = \   \P(T_1=y_1)\prod_{i=1}^{J-1}K(y_i,y_{i+1})\ = \ 
    \P(T_1=y_1)\prod_{i=1}^{n-1}K(y_i,y_{i+1}).
\end{eqnarray*}

We conclude therefore that ${\bf Y}$ is a Markov chain with kernel $K$. \qed

\subsubsection{Proof of Part (3)}  Recall the definitions of the increasing random sequences ${\bf V}$ and ${\bf W}$ with $V_1=W_1=1$ (cf. \eqref{V_def}), and ${\bf P^V}$ and ${\bf P^W}$.  For each realization, ${\bf V}$ and ${\bf W}$ are functions of the Markov sequence ${\bf T}$.  Therefore, conditional on ${\bf T}$ given the possible trajectory ${\bf t}$ with respect to ${\bf T}$, it follows immediately from the proved Part (1) that ${\bf P^V}\bigr | {\bf T} = {\bf t}$ and ${\bf P^W}\bigr | {\bf T}={\bf t}$ are RAMs. \qed

\subsubsection{Proof of Part (4)}  If ${\bf P}$ is a RAM, we have $\sum_{i\geq 1} P^V_i = \sum_{i\geq 1}P_i = 1$ a.s. or $\sum_{i\geq 1}P^W_i = 1$ a.s. respectively. Hence, in the two situations, we need only show the associated fractions ${\bf X^V}$ or ${\bf X^W}$ are conditionally independent or iid to deduce, respectively, that ${\bf P^V}\bigr | {\bf Y}={\bf y}$ is a RAM or ${\bf P^W}\bigr | T_1=t_1$ is a self-similar RAM.  We consider first the claim for ${\bf P^V}$, before discussing the statement for ${\bf P^W}$ at the end.

Let ${\bf y}$ be a possible sequence with respect to ${\bf Y}$, and associate to ${\bf y}$ the time $J$ as in the proof of part (2).
With respect to fixed times $V_{i+1}-V_i = l_i\in {\mathbb N}$ for $1\leq i<J$, noting \eqref{2_seq}, we have for $m\leq n<J$ that
\begin{eqnarray}
\label{part4_eqn}
&&\P\left( Y_i = y_i: 1\leq i\leq n +1, \ \ {\rm and \ \ }  V_{i+1}-V_i = l_i: 1\leq i\leq m\right) \nonumber\\
&& = \P\left(T_1=y_1\right)\prod_{i=1}^{m} Q_{y_i,y_i}^{l_i -1}Q_{y_i,y_{i+1}}\prod_{i=m+1}^n \sum_{l_i=1}^\infty Q^{l_i-1}_{y_i,y_i}Q_{y_i,y_{i+1}}\nonumber\\
&&= \P\left(T_1=y_1\right)\prod_{i=1}^{n}\frac{Q_{y_i,y_{i+1}}}{1-Q_{y_i,y_i}} \prod_{i=1}^{m} Q_{y_i,y_i}^{l_i -1}(1-Q_{y_i,y_i})\nonumber\\
&& = \P\left(Y_i=y_i: 1\leq i\leq n+1\right) \prod_{i=1}^{m} Q_{y_i,y_i}^{l_i -1}(1-Q_{y_i,y_i}).
\end{eqnarray}

Suppose $J<\infty$, and define $l_J=\infty$. For $n\geq J$, noting the calculation after \eqref{2_seq}, write
\begin{eqnarray*}
&&\P\left( Y_i = y_i: 1\leq i\leq n +1, \ \ {\rm and \ \ }  V_{i+1}-V_i = l_i: 1\leq i\leq J\right) \nonumber\\
&&= \P\left( Y_i = y_i: 1\leq i\leq J, \ \ {\rm and \ \ }  V_{i+1}-V_i = l_i: 1\leq i\leq J\right) \nonumber \\
&& = \P\left(T_1=y_1\right)\prod_{i=1}^{J-1} Q_{y_i,y_i}^{l_i -1}Q_{y_i,y_{i+1}}\nonumber\\
&&= \P\left(T_1=y_1\right)\prod_{i=1}^{J-1}\frac{Q_{y_i,y_{i+1}}}{1-Q_{y_i,y_i}} \prod_{i=1}^{J-1} Q_{y_i,y_i}^{l_i -1}(1-Q_{y_i,y_i})\nonumber\\
&& = \P\left(Y_i=y_i: 1\leq i\leq n+1\right) \prod_{i=1}^{J-1} Q_{y_i,y_i}^{l_i -1}(1-Q_{y_i,y_i}).
\end{eqnarray*}

Recall \eqref{Ram_2}, and consider the variables ${\bf X^V}=\{X^V_j\}_{j\geq 1}$ where 
\begin{align}
\label{V_ram}
X^V_j  =\left\{\begin{array}{cl}
\sum_{i=V_j}^{V_{j+1}-1}X_i\prod_{l=V_j}^{i-1}(1-X_l)&\\
\ \ \ \ \ \ \ \ \ =1-\prod_{i=V_j}^{V_{j+1}-1}(1-X_i) & {\rm if \ } V_j<\infty\\
 1 & {\rm if \ } V_j=\infty.\end{array}\right.
\end{align}
When ${\bf X}$ is composed of iid variables, that is ${\bf P}$ is a self-similar RAM, we will argue now that the fractions
${\bf X^V}\bigr | {\bf Y}={\bf y}$ form a conditionally independent sequence, and therefore
${\bf P^V}\bigr|{\bf Y}={\bf y}$ is RAM. We split into subcases, $J=\infty$ versus $J<\infty$.

When $J=\infty$, let $r\geq n\geq 1$, and $\langle \alpha_i: 1\leq i\leq n\rangle\in (0,1)^n$. Write
\begin{align}
 & \P\left(1-X^V_j\leq\alpha_j: 1\leq j\leq n\biggr| Y_j = y_j: 1\leq j\leq r+1 \right)\nonumber\\
&= \sum_{l_1=1}^\infty \cdots \sum_{l_n=1}^\infty
\P\Bigr(1-X^V_j\leq\alpha_j: 1\leq j\leq n \biggr|\nonumber\\
& \ \ \ \ \ \ \ \ \ \ \ \ \ \ \ \ \ \ \ \  \left\{ V_{i+1}-V_i = l_i, 1\leq i\leq n\right\}
\cap \left\{ Y_i = y_i: 1\leq i\leq r+1 \right\}\Bigr)\nonumber\\
&\ \ \ \ \ \ \ \ \ \ \ \ \ \  \times \P\left(V_{i+1}-V_i = l_i, 1\leq i\leq n
\biggr|Y_j = y_j: 1\leq j\leq r+1 \right).
\label{1_seq}
\end{align}

Relative to $\{l_j\}_{j=1}^n$, define the sequence ${\bf u}=\{u_j\}_{j=1}^{n+1}$ where $u_1=1$ and $u_j = 1+\sum_{k=1}^{j-1} l_k$ for $2\leq j\leq n+1$, which marks the first $n$ times when ${\bf Y}$ changes states.  In particular, on the event $\big\{ V_{i+1}-V_i = l_i, 1\leq i\leq n\big\}$, we have
$V_j = u_j$ for $1\leq j\leq n+1$. Given this event, from \eqref{Ram_2}, the fractions $\{X^V_j\}_{j=1}^n$ satisfy
$1-X^V_j = \prod_{k=u_j}^{u_{j+1}-1} (1-X_k)$ for $1\leq j\leq n$ and are independent, no longer depending on ${\bf Y}$.  The last display \eqref{1_seq}, noting \eqref{part4_eqn}, equals
\begin{align}
& \sum_{l_1=1}^\infty \cdots \sum_{l_{n}=1}^\infty \P\left(\prod_{k=u_j}^{u_{j+1}-1}(1-X_k)\leq\alpha_j: 1\leq j\leq n\right)\left[\prod_{j=1}^nQ_{y_jy_j}^{l_j-1}(1-Q_{y_jy_j})\right]\nonumber\\
= & \sum_{l_1=1}^\infty \cdots \sum_{l_{n}=1}^\infty  \prod_{j=1}^n \P\left(\prod_{k=u_j}^{u_{j+1}-1} (1-X_k)\leq\alpha_j\right)Q_{y_jy_j}^{l_j-1}(1-Q_{y_jy_j})\nonumber\\
= & \sum_{l_1=1}^\infty \cdots \sum_{l_{n}=1}^\infty \prod_{j=1}^n \P\left(\prod_{k=1}^{l_j}(1-X_k)\leq\alpha_j\right)Q_{y_jy_j}^{l_j-1}(1-Q_{y_jy_j})\nonumber\\
=&  \prod_{j=1}^n \sum_{l_j=1}^\infty \P\left(\prod_{k=1}^{l_j}(1-X_k)\leq\alpha_j\right)Q_{y_jy_j}^{l_j-1}(1-Q_{y_jy_j}),
\label{fractions X V}
\end{align}
in factored form.  Therefore, the fractions ${\bf X^V}$ are conditionally independent as desired and ${\bf P^V}\bigr | {\bf Y}={\bf y}$ is a RAM in the case $J=\infty$.

When $J<\infty$, note that the collection $\{X^V_j|{\bf Y=y}\}_{j\geq J}$ is a deterministic sequence of 1s. Thus, we need only show that the proportions $\{X^V_j|{\bf Y=y}\}_{j=1}^{J-1}$ are independent. Define $l_J=\infty$ and for $n\geq J$, write that
\begin{align}
 & \P\left(1-X^V_j\leq\alpha_j: 1\leq j<J\biggr| Y_j = y_j: 1\leq j\leq n+1\right)\nonumber\\
= &\sum_{l_1=1}^\infty \cdots \sum_{l_{J-1}=1}^\infty
\P\Bigr(1-X^V_j\leq\alpha_j: 1\leq j<J \biggr|\nonumber\\
&\ \ \ \ \ \ \ \ \ \ \ \ \ \ \ \ \ \ \ 
\left\{ V_{i+1}-V_i = l_i, 1\leq i\leq J\right\}
\cap \left\{ Y_i = y_i: 1\leq i\leq n+1\right\}\Bigr)\nonumber\\
&\ \ \ \ \ \ \ \ \ \ \ \  \times \P\left(
 V_{i+1}-V_i = l_i, 1\leq i\leq J\biggr|Y_j = y_j: 1\leq j\leq n+1\right).
\label{1_seqprime}
\end{align}

Define for $j\leq J+1$ the sequence $u_j$ as before, and note $u_{J+1}=\infty$. One derives similarly, noting the calculation after \eqref{part4_eqn}, that the last display \eqref{1_seqprime} equals
$$\prod_{j=1}^{J-1} \sum_{l_j=1}^\infty \P\left(\prod_{k=1}^{l_j}(1-X_k)\leq\alpha_j\right)Q_{y_jy_j}^{l_j-1}(1-Q_{y_jy_j}),$$
in factored form. Therefore, the fractions ${\bf X^V}$ are conditionally independent as desired and ${\bf P^V}\bigr | {\bf Y}={\bf y}$ is a RAM also in the case $J<\infty$.

We now aim to show when ${\bf P}$ is a self-similar RAM and $t_1$ is a recurrent state for ${\bf T}$ that ${\bf P^W}|T_1=t_1$ is a self-similar RAM. As $t_1$ is a recurrent state with respect to ${\bf T}$, almost surely the sequence ${\bf W}$ does not take on the value $\infty$. Consider the variables ${\bf X^W}=\{X^W_j\}_{j\geq 1}$ where 
\begin{align*}
X^W_j  = \left\{\begin{array}{rl}\sum_{i=W_j}^{W_{j+1}-1}X_i\prod_{l=W_j}^{i-1}(1-X_l) & {\rm if \ } W_j<\infty\\
  1 & {\rm if \ } W_j=\infty. \end{array}\right.
\end{align*}
Then, noting \eqref{Ram_2}, almost surely, $X^W_j=\sum_{i=W_j}^{W_{j+1}-1}X_i\prod_{l=W_j}^{i-1}(1-X_l)$.

Following the above argument, with respect to ${\bf X^V}$ when $J=\infty$, we arrive at the equation
\begin{align*}
&\P\left(1-X^W_j\leq\alpha_j: 1\leq j\leq n\biggr| T_1=t_1\right)\\
 = &\sum_{l_1=1}^\infty \cdots \sum_{l_n=1}^\infty  \prod_{j=1}^n \P\Big(\prod_{k=1}^{l_j}(1-X_k)\leq\alpha_j\Big)\P\left(W_{j+1}-W_j = l_i: 1\leq i\leq n\bigr | T_1=t_1\right).
\end{align*}
But, given $T_1=t_1$, the variables $\{W_{j+1}-W_j\}_{j\geq 1}$ are iid cycle lengths of the Markov chain.  Hence,
the last display equals
$$\prod_{j=1}^n \sum_{l_j=1}^\infty \P\Big(\prod_{k=1}^{l_j}(1-X_k)\leq\alpha_j\Big) \P\left ( W_2 - W_1 = l_j\bigr | T_1=t_1\right),$$
indicating the fractions ${\bf X^W}$ are conditionally iid, and therefore
${\bf P^W}\bigr|T_1=t_1$ is a self-similar RAM.  
\qed

\subsection{Proof of Theorem \ref{Gem to MCcGEM}: GEM to MCcGEM}

Let ${\bf P}= \langle P_i: i\geq 1\rangle$ be a GEM$(\theta)$ sequence, with respect to corresponding iid Beta$(1,\theta)$ proportions ${\bf X}=\{X_j\}_{j\geq 1}$.  Also, let ${\bf T}$ be an independent Markov chain on $\mathcal{X}$, starting from distribution $\mu$, with homogeneous kernel $Q$.

In Part (2) of Theorem \ref{RAM}, we showed that the associated sequence ${\bf Y}$ is a Markov chain with transition kernel $K$ on $\X$ such that
\begin{align*}
K(z,w)  =\left\{\begin{array}{rl}
\frac{Q_{z,w}}{1-Q_{z,z}}  & {\rm when \ } z\neq w\ {\rm and \ } Q_{z,z}\neq 1\\
 1  & {\rm when \ } z=w\ {\rm and \ } Q_{z,z}=1\\
 0  & {\rm otherwise. \ }
\end{array}\right.
\end{align*}
   By inspection, the kernel $K = K_{G}$, in the definition of the MCcGEM distribution \eqref{MCcGEM_kernel}, where $G=\theta(Q-I)$.

Recall now the switch times ${\bf V}$ with respect to the chain ${\bf T}$ (cf. \eqref{V_def}).  In Part (4) of Theorem \ref{RAM}, as $P$ is a self-similar RAM, we proved that ${\bf P^V}$, conditional on ${\bf Y}$, is a RAM.  In particular, we showed that the associated fractions ${\bf X^V}=\{X^V_j\}_{j\geq 1}$, given ${\bf Y}$, are independent variables.  Hence, to identify the joint distribution of $({\bf P^V}, {\bf Y})$, we need only find the conditional distribution of each fraction $X^V_j\bigr |{\bf Y}$, for $j\geq 1$.
 
 To this end, let ${\bf y}$ be a possible sequence for ${\bf Y}$. Associate $J=\sup\{j:y_j\neq y_{j-1}\}$ to ${\bf y}$ as before. Recall from \eqref{V_ram} that $X^V_j = 1-\prod_{k=V_j}^{V_{j+1}-1} (1-X_k)$ for $j\leq J$. Write, for $j<\min\{J,n\}$ and $m\geq 1$,
\begin{align*}
&\E\left[(1-X^V_j)^m\biggr|Y_i=y_i:1\leq i\leq n\right] \\
=&\E\Bigg[\E\Bigg[\prod_{k=V_j}^{V_{j+1}-1}(1-X_k)^m\biggr|Y_i=y_i, V_{i+1}-V_i: 1\leq i\leq n\Bigg]\biggr|Y_i=y_i: 1\leq i\leq n\Bigg].
\end{align*}

  Note now, if $Z$ is a Beta$(1,\alpha)$ random variable, then E$[(1-Z)^m]=\frac\alpha{\alpha+m}$.
	Then, by the independence of ${\bf X}$ and ${\bf T}$, noting from \eqref{part4_eqn} that
	$\P(V_{j+1}-V_j=\ell|Y_i = y_i: 1\leq i\leq n) = Q_{y_i,y_i}^{\ell-1}(1-Q_{y_i,y_i})$, the above display equals
  \begin{align*}
 &\E\left[\left(\frac\theta{\theta+m}\right)^{V_{j+1}-V_j}\biggr|Y_i=y_i: 1\leq i\leq n\right]\\
 & \ \ =\sum_{l=1}^\infty\left(\frac\theta{\theta+m}\right)^lQ_{y_j,y_j}^{l-1}(1-Q_{y_j,y_j}) \ = \ \frac{\theta(1-Q_{y_j,y_j})}{\theta(1-Q_{y_j,y_j})+m}.
\end{align*}

Thus, we see that $X^V_j\biggr | {\bf Y}={\bf y}$ is a Beta$(1,\theta(1-Q_{y_j,y_j}))$ random variable when $j<J$. 

When $J\leq j<\infty$, recall that $y_j$ is an absorbing state, and so $Q_{y_j,y_j}=1$ and $X_j^V=1$. Thus $X_j^V|{\bf Y=y}\sim$ Beta$(1,0)=$ Beta$(1,\theta(1-Q_{y_j,y_j}))$. 

Then, for all $j\geq 1$, we see that $X^V_j\biggr | {\bf Y}={\bf y}$ is a Beta$(1,\theta(1-Q_{y_j,y_j}))$ random variable.  Hence ${\bf P^V}\bigr | {\bf  Y}={\bf y}$ is a disordered GEM with parameters $\theta(1-Q_{y_j,y_j}) = -G_{y_j,y_j}$ for $j\geq 1$.  Therefore, we conclude that $({\bf P^V}, {\bf Y})$ has 
 a MCcGEM$(\theta(Q-I))$ distribution with respect to $\mu$. \qed

\subsection{Proof of Theorem \ref{inhom to MCcGEM}:  Time inhomogeneous MC to MCcGEM}

We first specify certain asymptotics which will be helpful, before going to the main body of the proof in Subsection \ref{completion}.

\begin{lemma}\label{f1}
For $\gamma>0$ and integers $1\leq m\leq n$, let
$$f_m^n(\gamma)=\prod_{j=m+1}^n\left(1-\frac\gamma j\right).$$
 Then, for $0<a<b$ and integers $M\geq 0$, we have
 $$\lim\limits_{n\rightarrow\infty}f_M^n(\gamma)n^\gamma=\frac{\Gamma(M+1)}{\Gamma(M+1-\gamma)} \ \ \text{and} \ \ \lim\limits_{n\rightarrow\infty}f_{\lfloor an\rfloor}^{\lfloor bn\rfloor}(\gamma)\left ( \frac{b}{a}\right)^\gamma =1.$$
\end{lemma}
\begin{proof} Write
$$f_l^n(\gamma)=\prod_{j=l+1}^n\left(1-\frac\gamma j\right)=\frac{\prod_{j=l}^n(j-\gamma)}{\prod_{j=l}^n j}=\frac{\Gamma(n+1-\gamma)\Gamma(l+1)}{\Gamma(n+1)\Gamma(l+1-\gamma)}.$$
By Stirling's approximation, 
for $u,v\in \mathbb{R}$, we have 
$\frac{\Gamma(n+u)}{\Gamma(n+v)}n^{v-u}\rightarrow 1$ as $n\rightarrow\infty$,
from which the desired asymptotics follow immediately.
\end{proof}

\begin{prop}\label{f2} Let $r\geq 1$ be an integer. Let also $\{a_i\}_{j=1}^r$, $\{b_i\}_{i=1}^r$, and $\{\gamma_i\}_{i=1}^r$ be collections of positive numbers such that $a_j<b_j$ for $1\leq j\leq r$. Then,
$$\lim_{s_0\rightarrow\infty}\sum_{s_1=\lceil a_1s_0\rceil}^{\lceil b_1s_0\rceil-1}\cdots\sum_{s_r=\lceil a_rs_{r-1}\rceil}^{\lceil b_rs_{r-1}\rceil-1}
\left[\prod_{j=1}^r{s_j}^{-1}f_{s_j}^{\lfloor b_js_{j-1}\rfloor -1}(\gamma_j)\right] \ = \ \prod_{j=1}^r\gamma_j^{-1}\Big(1-\Big(\frac{a_j}{b_j}\Big)^{\gamma_j}\Big).$$
\end{prop}

\begin{proof}
The argument follows by inputting the asymptotics in Lemma \ref{f1}.  We show only the case $r=1$, as the extension to $r>1$ is straightforward.

Again, by Stirling's approximation, for each $u,v\in\mathbb{R}$,
$\lim_{n\rightarrow\infty}\frac{\Gamma(n+u)}{\Gamma(n+v)}n^{v-u} =  1$.
Then, for $\epsilon>0$ and all large $n$, we have
$$(1-\epsilon)n^{u-v}\leq\frac{\Gamma(n+u)}{\Gamma(n+v)}\leq(1+\epsilon)n^{u-v}.$$
Hence, for $\epsilon, a,b,\gamma>0$ with $a<b$, and sufficiently large $n$, we estimate
\begin{align}
\label{f2_eq}
 & (1-\epsilon)^2\sum_{s=\lfloor an\rfloor}^{\lfloor bn\rfloor-1}\lfloor bn\rfloor^{-\gamma}s^{\gamma-1} \nonumber\\
 & \leq\sum_{s=\lfloor an\rfloor}^{\lfloor bn\rfloor-1}\frac{\Gamma(\lfloor bn\rfloor -\gamma)\Gamma(s)}{\Gamma(\lfloor bn\rfloor)\Gamma(s+1-\gamma)} \ =\ \sum_{s=\lfloor an\rfloor}^{\lfloor bn\rfloor-1}s^{-1}f_{s}^{\lfloor bn\rfloor -1}(\gamma) \nonumber\\
 & \leq \ (1+\epsilon)^2\sum_{s=\lfloor an\rfloor}^{\lfloor bn\rfloor-1}\lfloor bn\rfloor^{-\gamma}s^{\gamma-1}.
\end{align}

Now, by the monotonicity of $s^{\gamma-1}$, we have for $n>2/a$ that $\sum_{s=\lfloor an\rfloor}^{\lfloor bn\rfloor-1}s^{\gamma-1}$ is between the integrals $\int_{\lfloor an\rfloor-1}^{\lfloor bn\rfloor-1}s^{\gamma-1}ds$ and $\int_{\lfloor an\rfloor}^{\lfloor bn\rfloor}s^{\gamma-1}ds$.  We may compute
\begin{align*}
&\lim_{n\rightarrow\infty}\lfloor bn\rfloor^{-\gamma}\int_{\lfloor an\rfloor-1}^{\lfloor bn\rfloor-1}s^{\gamma-1}ds\\
&\ \ \ = \ \lim_{n\rightarrow\infty}\lfloor bn\rfloor^{-\gamma}\int_{\lfloor an\rfloor}^{\lfloor bn\rfloor}s^{\gamma-1}ds \ = \  
\frac1\gamma\left(1-\left(\frac ab\right)^\gamma\right).
\end{align*}
Then,
inserting into \eqref{f2_eq}, the proposition follows for $r=1$. 
\end{proof}

We now show a form of `weak ergodicity' for the Markov chain ${\bf T}$.
\begin{lemma}
\label{lem_weak}
For a generator matrix $G$, let $\theta>0$, and $M\geq 1$ be an integer, such that $Q:=I+G/\theta$ and $I+G/M$ are non-negative kernels on $\mathcal{X}$. Recall that $K_n = I +\frac{G}{n}{\mathbbm 1}(n>M)$ for $n\geq 1$ (cf. \eqref{K_n_eqn}).  Let $\pi$ be a stochastic vector and $\mu$ be a stationary distribution for $Q$ such that $\pi^tQ^n\rightarrow \mu^t$ entry-wise. Then, as $n\rightarrow\infty$, both (a) $\mu^n:= \pi^t\prod_{i=1}^nK_i\rightarrow \mu^t$, and (b) $\big(\mu^n\big)^tQ\rightarrow\mu^t$, hold entry-wise.
\end{lemma}

\begin{proof}  We separate into four steps.

 \vskip .1cm

{\it Step 1.} Fix an integer $m\geq \max(M,\theta)$ and write the stochastic matrix,
\begin{align*}
\prod_{j=m+1}^nK_j &=\prod_{j=m+1}^n\left[\left(1-\frac\theta j\right)I+\frac\theta jQ\right]\\
&=\left[\prod_{j=m+1}^n\left(1-\frac\theta j\right)\right]\left(I+\sum_{i=1}^{n-m}Q^i\sum_{m<j_1<\cdots <j_i\leq n}\prod_{l=1}^i\frac\theta{j_l-\theta}\right),
\end{align*}
as a polynomial in $Q$ with positive coefficients. 

\vskip .1cm
{\it Step 2.} We now show that any fixed degree coefficient of the polynomial vanishes as $n\rightarrow\infty$.   For each $i$, denote the $n$th coefficient of $Q^i$ by $[Q^i]_n$.
By Lemma \ref{f1}, $[Q^0]_n=f_m^n(\theta)\rightarrow0$ as $n\rightarrow\infty$.  Also, as $f^n_m(\theta)\sim n^{-\theta}$ by Lemma \ref{f1}, we have for $i\geq 1$ that
\begin{align*}
[Q^i]_n= & \left[\prod_{j=m+1}^n\left(1-\frac\theta j\right)\right]\sum_{m<j_1<...<j_i\leq n}\prod_{l=1}^i\frac\theta{j_l-\theta}\\
= & \ \theta^if_m^n(\theta)\sum_{j_1=m+1}^{n-i+1}\frac1{j_1-\theta}\sum_{j_2=j_1+1}^{n-i+2}\frac1{j_2-\theta}\cdots\sum_{j_i=j_{i-1}+1}^n\frac1{j_i-\theta}\\
\leq & \theta^if_m^n(\theta)\left[\ln\left(\frac{n-\theta}{m+1-\theta}\right)+\frac1{m+1-\theta}\right]^i\\
\leq & \ C(\theta, m)n^{-\theta}\left[\ln\left(\frac{n-\theta}{m+1-\theta}\right)\right]^i \xrightarrow{n\rightarrow\infty} 0.
\end{align*}

\vskip .1cm
{\it Step 3.} For each $x\in \X$, let $e_x$ denote the vector in $\mathbb{R}^\X$ with a $1$ in the entry corresponding to state $x$ and $0$'s elsewhere.
Since $Q$ is a stochastic kernel, observe for each $x\in \X$ and $n\geq m$ that
$$1= \sum_{z\in \X}e_x^t\left[\prod_{j=m+1}^n K_j\right] e_z = \sum_{i=0}^{n-m}[Q^i]_n \sum_{z\in \X}e_x^tQ^i e_z = \sum_{i=0}^{n-m}[Q^i]_n.$$

Also, as $\mu$ is a stationary eigenvector of $Q$, note that $\mu$ is also a stationary eigenvector of $\{K_n\}_{n\geq 1}$.  Recall that $\left(\pi-\mu\right)^tQ^n\rightarrow0$ as $n\rightarrow \infty$ entry-wise, and $\mu^m = \pi^t \prod_{i=1}^mK_i$.  Hence, $\left(\mu^m-\mu\right)^tQ^n\rightarrow0$ as $\prod_{i=1}^mK_i$ is 
a polynomial in $Q$. 

 With these observations, for each $x\in \X$ and positive integers $n$ and $R<n-m$, we may bound
\begin{align*}
\left|\mu^n_l-\mu_l\right| & =\left|(\mu^m-\mu)^t\left[\prod_{j=m+1}^nK_j\right]e_l\right|\\
 & =\Big|\sum_{i=0}^{n-m}[Q^i]_n(\mu^m-\mu)^tQ^ie_l\Big|\ \leq \ 
 \sum_{i=0}^R[Q^i]_n + \Big| \sum_{i=R+1}^{n-m} [Q^i]_n (\mu^m-\mu)^tQ^ie_l\Big|\\
&\leq \sum_{i=0}^R[Q^i]_n + \max_{r>R}\left|(\mu^m-\mu)^tQ^re_l\right|.
\end{align*}
 The last display converges by the calculation in Step 2 to
 $\max_{r>R}\big |\left(\mu^m-\mu\right)^TQ^re_l\big|$, as $n\rightarrow\infty$, and in turn
vanishes as $R\rightarrow\infty$.  Hence, the first limit follows.

\vskip .1cm
{\it Step 4.}  Finally, by Fatou's lemma, the proved first limit (a), and that $\mu$ is a stationary vector of $Q$, we have for each $j\in \X$ that
\begin{equation}
\label{Fatou}
\liminf_{n\rightarrow\infty} \big(\mu^n\big)^tQe_j = \liminf_{n\rightarrow\infty} \sum_{i\in \X} \mu^n_i Q_{i,j} \geq \sum_{i\in \X}\mu_i Q_{i,j} = \mu_j.
\end{equation}

Now, suppose for a particular $k\in \X$ that $\limsup_{n\rightarrow\infty} \big(\mu^n\big)^t Qe_k = L> \mu_k$.  Then, as $\big(\mu^n\big)^tQ$ is a stochastic vector, we would have for each $n\geq 1$ that
\begin{equation*}
\label{Fatou1}
1=\limsup_{n\rightarrow\infty}\sum_{l\in \X} \big(\mu^n\big)^tQe_l \geq L + \liminf_{n\rightarrow\infty}\sum_{l\neq k}\big(\mu^n\big)^tQe_l.
\end{equation*}
But, as $\mu$ is a stochastic vector and noting \eqref{Fatou}, we have by Fatou's lemma again that the last display is larger than $L+\sum_{l\neq k}\mu_l>1$, a contradiction, and the second limit (b) holds.
\end{proof}

\subsubsection{Completion of the proof of Theorem \ref{inhom to MCcGEM}}
\label{completion}
We will argue in a few steps.

\vskip .1cm
{\it Step 1.}
Recall the definition of kernel $G'$ (cf. \eqref{G'_eqn}). We now argue that $G'$ is a generator matrix:  As $\mu$ is a stationary vector of $Q$ and $G = \theta(Q-I)$, we have $\mu^t G = 0$ is the zero vector.   Since $G$ is a generator matrix, we have $G'_{i,j} = (\mu_j/\mu_i)G_{j,i}1(\mu_i\neq 0)\geq 0$ for $i\neq j$, and $\sum_{j}G'_{i,j} = \frac{1(\mu_i\neq 0)}{\mu_i}\sum_{j} \mu_jG_{j,i} = 0$.  Moreover,
$$\sup_i |G'_{i,i}| \ =\ \sup_i |G_{i,i}1(\mu_i\neq 0)|
\ \leq\ \sup_i |G_{i,i}|\ <\ \infty.$$

\vskip .1cm

{\it Step 2.}  Recall the Markov chain ${\bf T}$, with transition kernels $\{K_n = I+\frac{G}{n}1(n>M)\}_{n\geq 1}$ (cf. \eqref{K_n_eqn}), starting from $\pi$.  Recall the associated variable $N_n$ and sequence ${\bf P}_n$.

Now, for $i\geq N_n>j\geq 1$ define
\begin{equation}
\label{step2_star_star}
X_{n,j}=P_{n,j}/\left(1-\sum_{i=1}^{j-1}P_{n,i}\right)\hspace{1cm}\text{and}\hspace{1cm}X_{n,i}=1.
\end{equation}
  The variables ${\bf X}_n = \{X_{n,i}\}_{i\geq 1}$ are the associated fractions to the distribution ${\bf P}_n$ on $\N$ and, by Proposition \ref{oneidentity_prop}, for $j\geq 1$,
\begin{equation}
P_{n,j}=X_{n,j}\prod_{i=1}^{j-1}(1-X_{n,i}) \ \ {\rm and \ \ } 1-\sum_{i=1}^{j-1}P_{n,j} = \prod_{i=1}^{j-1}\left(1-X_{n,i}\right).
\label{step2_equation}
\end{equation}

For $j\geq 0$, also define
\begin{equation}
S_j=n\left(1-\sum_{i=1}^j P_{n,j}\right) = n\prod_{i=1}^{j}\left( 1-X_{n,i}\right).
\label{S_def}
\end{equation}
In terms of the switching times ${\bf V}$, and the first time $N_n$ that the chain ${\bf T}$ switches after time $n$, we have $S_0 = n$, $S_j = V_{N_n -j}-1$ for $1\leq j\leq N_n-1$, and $S_j=0$ for $j\geq N_n$.  Recall also that $\tau_{n,j}= nP_{n,j}$ for $j\geq 1$.  In words, $\{S_j\}$ are the times before time $n$ at which the chain switches states when considered in reverse order, and $\{\tau_{n,j}\}$ are the lengths of the associated sojourns in the figure below.

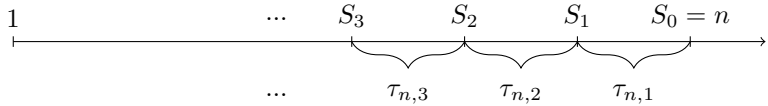
\begin{center}
\begin{tikzpicture}
\draw[->] (0,0) -- (10,0);
\node at (3.5,.3){...};
\node at (3.5,-.7){...};
\draw (0,-2pt) -- (0,2pt)node[anchor=south] {1};
\draw (4.5,-2pt) -- (4.5,2pt) node[anchor=south]{$S_3$};
\draw (6,-2pt) -- (6,2pt) node[anchor=south]{$S_2$};
\draw (7.5,-2pt) -- (7.5,2pt) node[anchor=south]{$S_1$};
\draw (9,-2pt) -- (9,2pt) node[anchor=south]{$S_0=n$};
\draw [decorate,decoration={brace,amplitude=10pt}]
(9,0) -- (7.5,0) node [midway,yshift=-0.7cm]{$\tau_{n,1}$};
\draw [decorate,decoration={brace,amplitude=10pt}]
(7.5,0) -- (6,0) node [midway,yshift=-0.7cm]{$\tau_{n,2}$};
\draw [decorate,decoration={brace,amplitude=10pt}]
(6,0) -- (4.5,0) node [midway,yshift=-0.7cm]{$\tau_{n,3}$};
\end{tikzpicture}
\end{center}

\vskip .1cm
{\it Step 3.}  Recall the sequence ${\bf Y}_n$ given in \eqref{Y_eqn}, where $Y_{n,j} = T_{V_{N_n-j}}$ for $1\leq j\leq N_n -1$ and $Y_{n,i}= T_1$ for $i\geq N_n$.
We now aim to compute the finite dimensional distributions of $({\bf P}_n, {\bf Y}_n)$ or equivalently of $({\bf X}_n,{\bf Y}_n)$.  
To this end, fix the integer $r\geq 1$, and consider numbers $\{\beta_j\}_{j=1}^r\in(0,1)^r$ such that $s_j:=n\prod_{i=1}^j(1-\beta_i)\in\mathbb{N}$, for $1\leq j\leq r$, are all integers. Set also $s_0=n$ and recall $S_0=n$.  

Note from \eqref{step2_equation} and \eqref{S_def} that 
\begin{eqnarray*}
X_{n,j}=\beta_j \ {\rm for \ } 1\leq j\leq r &\Longleftrightarrow& S_j=s_j=n\prod_{i=1}^j(1-\beta_i) \ {\rm  for \ } 1\leq j\leq r\\
&\Longleftrightarrow& \tau_{n,j}=nP_{n,j}= s_{j-1}-s_j \ {\rm for \ }1\leq j\leq r.
\end{eqnarray*}

Then, with respect to a possible sequence ${\bf y}$,
we have 
\begin{align}
\label{step3_eqn}
 & \text{P}\left(X_{n,j}=\beta_j, Y_{n,j}=y_{j}: 1\leq j\leq r\right)\\
= & P\left(\tau_{n,j}=s_{j-1}-s_j,\ Y_{n,j}=y_j: 1\leq j\leq r\right)\nonumber\\
= & \sum_{\stackrel{y_{r+1}\in \mathcal{X}:}{ y_{r+1}\neq y_r}}P(T_{s_r}=y_{r+1})\prod_{j=1}^rP\left(T_{s_{j-1}}=\cdots= T_{s_j+1}=y_j\bigr\vert  T_{s_j}=y_{j+1}\right).\nonumber
\end{align}

Note the computation for $M\leq l<n$ and $z\neq y$,
\begin{align*}
&P(T_n=\cdots = T_{l+1}=y| T_l=z)\\
&\ \ \ \ \ \ \ \ \ \ \ =\frac{G_{z,y}}{l}\prod_{j=l+1}^{n-1}\left(1+\frac{G_{y,y}}j\right)=\frac{G_{z,y}}{l}f_l^{n-1}\left(-G_{y,y}\right).
\end{align*}
Recall also that $\mu^s_y= P\left(T_s = y\right)$.  Since $G = \theta(Q-I)$, we observe
\begin{align*}
\sum_{\stackrel{y\in \mathcal{X}:}{ y\neq z}} \mu^{s}_{y} G_{y,z}
&= \theta \sum_{\stackrel{y\in \mathcal{X}:}{ y\neq z}} \mu^{s}_{y} \left(Q-I\right)_{y,z}\\
&= \theta \sum_{\stackrel{y\in \mathcal{X}:}{ y\neq z}} \mu^{s}_{y} Q_{y,z}  = 
 \theta \left[ \big(\mu^{s}\big)^t Qe_{z} - \mu_z^{s}Q_{z,z}\right].
 \end{align*}

Then, \eqref{step3_eqn} equals
\begin{align}
\label{step3_last}
 & \sum_{\stackrel{y_{r+1}\in \mathcal{X}:}{y_{r+1}\neq y_r}}\mu^{s_r}_{y_{r+1}}\prod_{j=1}^r\frac{G_{y_{j+1},y_j}}{s_j}f_{s_j}^{s_{j-1}-1}\left(-G_{y_j,y_j}\right)\\
= & \left[(\mu^{s_r})^tQe_{y_r}-\mu^{s_r}_{y_r}Q_{y_r,y_r}\right]\frac\theta{s_r}f_{s_r}^{s_{r-1}-1}\left(-G_{y_r,y_r}\right)\prod_{j=1}^{r-1}\frac{G_{y_{j+1},y_j}}{s_j}f_{s_j}^{s_{j-1}-1}\left(-G_{y_j,y_j}\right)\nonumber
\end{align}

\vskip .1cm

{\it Step 4.}
We now sum the display \eqref{step3_last} over all appropriate values of $\{s_j\}_{j=1}^r$ such that $0<X_{n,j}\leq\beta_j$ for $1\leq j\leq r< N_n$, where we recall $N_n$ is the time the chain switches after time $n$.  Then, we have from \eqref{S_def} 
that
\begin{equation}
\label{step4_eqn}
1\leq \tau_{n,j}=nP_{n,j}=S_{j-1}-S_j = X_{n,j}S_{j-1}.
\end{equation}
Moreover, also from \eqref{S_def},
we have $s_r\geq n\prod_{j=1}^r(1-\beta_j)$ diverges to infinity as $n\rightarrow\infty$. 

Recall $s_0=n$ and $\lim_{n\rightarrow\infty}N_n=\infty$ a.s. Then, with equation \eqref{step4_eqn} in hand,
\begin{align*}
& P\left(0<X_{n,j}\leq\beta_j,\ Y_{n,j}=y_j: 1\leq j\leq r\right)\nonumber \\
&=P\left(1\leq \tau_{n,j} =S_{j-1}-S_j\leq S_{j-1}\beta_j,\ Y_{n,j}=y_j: 1\leq j\leq r\right) \nonumber \\
&=P\left(S_{j-1}(1-\beta_j)\leq S_j\leq S_{j-1} -1,\ Y_{n,j}=y_j: 1\leq j\leq r\right) \nonumber \\
&=\sum_{s_1=\lceil s_0(1-\beta_1)\rceil}^{s_0-1}\cdots\sum_{s_r=\lceil s_{r-1}(1-\beta_r)\rceil}^{s_{r-1}-1} \left[(\mu^{s_r})^tQe_{y_r}-\mu^{s_r}_{y_r}Q_{y_r,y_r}\right]\nonumber \\
&\ \ \ \ \ \ \ \times \frac\theta{s_r}f_{s_r}^{s_{r-1}-1}\left(-G_{y_r,y_r}\right)\prod_{j=1}^{r-1}\frac{G_{y_{j+1},y_j}}{s_j}f_{s_j}^{s_{j-1}-1}\left(-G_{y_j,y_j}\right).
\end{align*}

\vskip .1cm
{\it Step 5.}
From \eqref{S_def},
the sum index $s_r\geq n\prod_{j=1}^r(1-\beta_j)$ diverges to infinity as $n\rightarrow\infty$.  Also, by Lemma \ref{lem_weak}, we have $\lim_{s\rightarrow\infty}\mu^s_y = \mu_y$ and $\lim_{s\rightarrow\infty}\big(\mu^s\big)^tQe_y = \mu_y$ for each $y\in \X$.  Therefore, as $n\rightarrow\infty$, we have
$$\theta\left[\big(\mu^{s_r}\big)^tQe_{y_r} - \mu^{s_r}_{y_r}Q_{y_r, y_r}\right] \rightarrow \theta \mu_{y_r}(1-Q_{y_r, y_r}) = \mu_{y_r}\big(-G_{y_r, y_r}\big).$$

Note that $-G_{i,i}>0$ for each $i\in \X$ since by assumption $G$ has no zero rows. Thus, by Proposition \ref{f2}, we have
\begin{align}
\label{step4_limit}
 & \lim_{n\rightarrow\infty}P\left(0<X_{n,j}\leq\beta_j,\ Y_{n,j}=y_j: 1\leq j\leq r\right) \\
= & \mu_{y_r}(-G_{y_r,y_r})\prod_{j=1}^{r-1}G_{y_{j+1},y_j}\nonumber \\
 & \ \ \ \ \ \ \times \lim_{n\rightarrow\infty}\sum_{s_1=\lceil s_0(1-\beta_1)\rceil}^{s_0-1}\cdots\sum_{s_r=\lceil s_{r-1}(1-\beta_r)\rceil}^{s_{r-1}-1}\prod_{j=1}^r{s_j}^{-1}f_{s_j}^{s_{j-1}-1}\left(-G_{y_j,y_j}\right)\nonumber\\
= & \mu_{y_r}(-G_{y_r,y_r})\left[\prod_{j=1}^{r-1}G_{y_{j+1},y_j}\right]\left[\prod_{j=1}^r\left(-G_{y_j,y_j}\right)^{-1}\left(1-(1-\beta_j)^{-G_{y_j,y_j}}\right)\right].\nonumber
\end{align}

Hence, if $\mu_{y_k}=0$ for some $1\leq k\leq r$, by bounding say $P(0<X_{n,j}\leq \beta_j, Y_{n,j} = y_j: 1\leq j\leq r) \leq P(0<X_{n,j}\leq \beta_j, Y_{n,j}=y_j: 1\leq j\leq k)$, the limit \eqref{step4_limit} vanishes.  Now, suppose that $\{y_j\}_{j=1}^r$ is such that $\mu_{y_j}>0$ for each $1\leq j\leq r$.  We may write the limit \eqref{step4_limit} as
\begin{align*}
& \mu_{y_1}\left[\prod_{j=1}^{r-1}\frac{\mu_{y_{j+1}}}{\mu_{y_j}}\frac{G_{y_{j+1},y_j}}{-G_{y_j,y_j}}\right]\left[\prod_{j=1}^r\left(1-(1-\beta_j)^{-G_{y_j,y_j}}\right)\right]\\
= & \mu_{y_1}\left[\prod_{j=1}^{r-1}\frac{G'_{y_j,y_{j+1}}}{-G'_{y_j,y_j}}\right]\left[\prod_{j=1}^r\left(1-(1-\beta_j)^{-G'_{y_j,y_j}}\right)\right],
\end{align*}
decomposed as a product of (a) the transition probability of the chain ${\bf Y}'$, with kernel $K_{G'}$ (cf. \eqref{MCcGEM_kernel}) and initial distribution $\mu$, running through states $\{y_j\}_{j=1}^r$, and of (b) the distribution functions of independent Beta$(1,-G'_{y_j, y_j})$ random variables for $1\leq j\leq r$.  Hence, the finite dimensional distributional convergence of $({\bf P}_n, {\bf Y}_n)$ as $n\rightarrow\infty$ is established. \qed

\subsection{Proof of Theorem \ref{occ cor}:  Occupation laws to MCcGEM and stick-breaking measures}
\label{proof_occ_cor1}
Consider the pairs $\{({\bf P}_n, {\bf Y}_n)\}_{n\geq 1}$, $({\bf P}', {\bf Y}')$ and $({\bf P}^+, {\bf T}')$ in the setting of Theorems \ref{inhom to MCcGEM} and \ref{occ cor}.  These objects belong to $[0,1]^\N\times \mathcal{X}^\N$.  We now discuss the topology on this space and its relatives, before going to the proof of \eqref{occ_cor_eqn} in Subsection \ref{proof_occ_cor}.
\medskip

\subsubsection{Topology} We endow the space $[0,1]^\mathbb{N}$ with a standard product metric $\rho^1$ and $\sigma$-field, generated in terms of this metric, which yields the usual product $\sigma$-field built from the Borel $\sigma$-fields on copies of $[0,1]$:
For 
$p,p'\in[0,1]^\mathbb{N}$,
$$\rho^1(p,p')=\sum_{n=1}^\infty2^{-n}|p_j-p'_j|.$$

Consider now the metric $\rho$ on 
$[0,1]^\mathbb{N}\times \mathcal{X}^\mathbb{N}$ defined as follows:
For $(p,y), (p',y')\in [0,1]^\N\times \mathcal{X}^\N$,
$$\rho((p,y),(p',y'))=\sum_{n=1}^\infty2^{-n}\left[|p_j-p'_j|+|y_j-y'_j|\right].$$
The corresponding $\sigma$-field on $[0,1]^\mathbb{N}\times \mathcal{X}^\mathbb{N}$, generated by $\rho$, is the usual product $\sigma$-field formed from the Borel $\sigma$-fields on copies of $[0,1]$ and $\mathcal{X}$.  Importantly, weak convergence of probability measures on $[0,1]^\N\times \X^\N$ translates to finite dimensional convergence of these laws.  Moreover, $([0,1]^\N\times \X^\N, \rho)$ is a complete, separable metric space.

Recall that $\Delta_\infty$ is the collection of all probabilities on $\N$:
$$\Delta_\infty=\left\{p\in[0,1]^{\mathbb{N}}:\sum_{j=1}^\infty p_j=1\right\}.$$
Since
$$\Delta_\infty=\bigcap_{n=1}^\infty\bigcap_{M=1}^\infty\bigcup_{m=M}^\infty\left\{p\in[0,1]^{\mathbb{N}}:1-\frac1n\leq\sum_{j=1}^mp_j\leq1+\frac1n\right\},$$
$\Delta_\infty\times \mathcal{X}^{\mathbb{N}}$ is a measurable set in $[0,1]^{\mathbb{N}}\times \mathcal{X}^{\mathbb{N}}$. 
We may endow $\Delta_\infty\times \mathcal{X}^\mathbb{N}$ with the restriction of the metric $\rho$ and the $\sigma$-field generated from the associated metric topology.  

For a fixed point $(p',y')\in\Delta_\infty\times \mathcal{X}^{\mathbb{N}}$, the projection map
$f:[0,1]^{\mathbb{N}}\times \mathcal{X}^{\mathbb{N}}\rightarrow\Delta_\infty\times \mathcal{X}^{\mathbb{N}}$, given by
$$f(p,y)=\left\{\begin{array}{cc}(p,y) & :(p,y)\in\Delta_\infty\times \mathcal{X}^{\mathbb{N}}\\ (p',y') & :(p,y)\notin\Delta_\infty\times \mathcal{X}^{\mathbb{N}}\end{array}\right.,$$
is measurable, and also continuous on the subset $\Delta_\infty\times \mathcal{X}^{\mathbb{N}}$. 

Now, denote the collection of probabilities on $\mathcal{X}$,
$$\Delta_\X=\left\{p\in[0,1]^\X:\sum_{l\in \X}p_l=1\right\},$$
and endow it with the metric $\rho^2(p,p') = \sum_{n\geq 1} 2^{-n} |p_n - p'_n|$,  and the associated Borel $\sigma$-field.
Define
$g:\Delta_\infty\times \X^{\mathbb{N}}\rightarrow\Delta_\X$ by
$$g((p,y))=\left\langle\sum_{j=1}^\infty p_j\mathbbm{1}_l(y_j): l\in \X\right\rangle.$$
Then, $g$ is a continuous and therefore measurable function on $\Delta_\infty\times \X^\N$:  Indeed, if $\{(p^n, y^n)\}_{n\geq 1}$ and $(p,y)$ belong to $\Delta_\X\times \X^\N$, and the finite dimensional convergence $(p^n,y^n)\rightarrow (p,y)$ holds, for each $l\in \X$, we have $\sum_{j\geq A} p^n_j\mathbbm{1}_l(y^n_j) \leq \sum_{j\geq A}p^n_j = 1-\sum_{j<A}p^n_j \stackrel{n\rightarrow\infty}{\longrightarrow} 1-\sum_{j<A}p_j$.  The claim now follows since (1) $\sum_{j<A}p^n_j \mathbbm{1}_l(y^n_j) \stackrel{n\rightarrow\infty}{\longrightarrow} \sum_{j<A}p_j\mathbbm{1}_l(y_j) \stackrel{A\rightarrow\infty}{\longrightarrow} g((p,y))$, and (2) $\sum_{j\geq A}p_j \stackrel{A\rightarrow\infty}{\longrightarrow} 1$.

\subsubsection{Proof of \eqref{occ_cor_eqn}}  
\label{proof_occ_cor}
First, we verify that the pairs $\{({\bf P}_n,{\bf Y}_n)\}_{n\geq 1}$,  $({\bf P}', {\bf Y}')$ and $({\bf P}^+, {\bf T}')$ belong almost surely to $\Delta_\infty\times \X^N$.  Clearly, $\{({\bf P}_n,{\bf Y}_n)\}_{n\geq 1}$ surely lives in 
$\Delta_\infty\times\X^{\mathbb{N}}$ by construction.  Also,  $({\bf P}', {\bf Y}')$ and $({\bf P}^+, {\bf T}')$ lie almost surely in $\Delta_\infty\times\X^{\mathbb{N}}$ since, by Theorem \ref{inhom to MCcGEM} and the assumptions of Theorem \ref{occ cor}, we have that ${\bf P}'$ and ${\bf P}^+$ are RAMs, and so
$\sum_{j=1}^\infty P'_j\stackrel d=\sum_{j=1}^\infty \hat P^+_j\stackrel{a.s.}=1$.

Now, from the finite dimensional or in other words weak convergence of $({\bf P}_n, {\bf Y}_n)$ to $({\bf P}', {\bf Y}')$ in Theorem \ref{inhom to MCcGEM}, we have $\nu_n=g\big(({\bf P}_n, {\bf Y}_n)\big)= g\circ f\big(({\bf P}_n, {\bf Y}_n)\big)$ converges weakly to $\nu = g\circ f\big(({\bf P}', {\bf Y}')\big)$ by the continuous mapping theorem, and so the left equality in \eqref{occ_cor_eqn} holds.  

On the other hand, with respect to $({\bf P}^+, {\bf T}')$, define ${\bf P^{+,V}}$ and ${\bf Y^+}$ as in the setting of Theorem \ref{Gem to MCcGEM}.  Recall that ${\bf T}'$ is a Markov chain with kernel $Q' = I + G'/\theta$ and initial stationary distribution $\mu$.  Then, by Theorem \ref{Gem to MCcGEM}, noting that $G' = \theta(Q'-I)$, we have that $({\bf P^{+,V}},  {\bf Y^+})$ has a MCcGEM$(G')$ distribution.  Hence, $({\bf P^{+,V}},  {\bf Y^+}) \stackrel d= ({\bf P}', {\bf Y}')$.  Since almost surely, by `unclumping',
$$g\circ f\big(({\bf P^{+,V}},  {\bf Y^+})\big) = g\circ f\big(({\bf P}^+, {\bf T}')\big) = \sum_{j\geq 1}P^+_j1_j(T'_j),$$
 we have
$g\circ f\big(({\bf P}', {\bf Y}')\big) \stackrel d= g\circ f\big(({\bf P}^+, {\bf T}')\big)$, and the right equality of \eqref{occ_cor_eqn} holds.
\qed

\subsection{Proof of Theorem \ref{occ cor2}:  Stick-breaking measures to Occupation laws}

The claim follows from Theorem \ref{occ cor} once we verify that a homogeneous Markov chain with kernel $\tilde Q$ and a homogeneous Markov chain with kernel $(\tilde Q')'=I+(\tilde G')'/\theta$, each with initial distribution $\mu$, are equivalent in distribution.

To this end, for any generator matrix $\tilde G=\theta(\tilde Q-I)$ and associated stationary distribution $\mu$, we observe that $(\tilde G')'_{ij}=\tilde G_{ij}$ when $\mu_i$ and $\mu_j$ are both positive:
$$(\tilde G')'_{ij}=\frac{\mu_j}{\mu_i}\tilde G'_{ji}\mathbbm{1}(\mu_i\neq 0)=\frac{\mu_j}{\mu_i}\frac{\mu_i}{\mu_j}\tilde G_{ij}\mathbbm{1}(\mu_i\neq 0)\mathbbm{1}(\mu_j\neq 0)=\tilde G_{ij}\mathbbm{1}(\mu_i\neq 0)\mathbbm{1}(\mu_j\neq 0).$$
Since $\tilde Q = I + \tilde G/\theta$ and $(\tilde Q')' = I + (\tilde G')'/\theta$, we conclude that $\tilde Q_{ij}=(\tilde Q')'_{ij}$ when $\mu_i$ and $\mu_j$ are both positive.

Finally, as $\mu$ is a stationary distribution, $\mu$ is only positive on positive recurrent states and for each recurrence class of $\tilde Q$, $\mu$ either assigns $0$ weight to each state in that class or strictly positive weights to each state in that class.  Hence, homogeneous Markov chains with kernels $\tilde Q$ and $(\tilde Q')'$, starting from $\mu$, are equal in distribution. \qed

\subsection{Proof of Theorem \ref{selfsimgem}: Type of self-similarity}
\label{self_similar}
We first give a proof of Lemma \ref{unique}, before going to the main argument in Subsection \ref{self_sim_comp}

\subsubsection{Proof of Lemma \ref{unique}}
Let $\{(\eta_j,X_j)\}_{j\geq 1}$ be i.i.d. copies of $(\eta,X)$, independent of $(\eta,X)$, all on a common probability space.

\vskip .1cm
{\it Existence:} Let $\chi(\cdot)=\sum_{j=1}^\infty \eta_j(\cdot)X_j\prod_{i=1}^{j-1}(1-X_i)$. Since $\P(X=0)<1$, we have $\prod_{j\geq 1}(1-X_j) =0$ a.s., and so $\big\langle X_j\prod_{i=1}^{j-1}(1-X_i):{j\geq 1}\big\rangle$ is a RAM.  Hence, $\chi$ is a random probability measure on ${\mathcal A}$ as $\chi({\mathcal A})=\sum_{j=1}^\infty X_j\prod_{i=1}^{j-1}(1-X_i)\stackrel{a.s.}=1$.  Moreover, \eqref{self_eqn} holds straightfowardly:
$$\chi=X_1\eta_1+(1-X_1)\left[\sum_{j=2}^\infty \eta_j(\cdot)X_j\prod_{i=2}^{j-1}(1-X_i)\right]\stackrel d=X_1\eta_1+(1-X_1) \tilde \chi,$$
where $\tilde \chi$ has the same law as $\chi$ and is independent of $(X_1,\eta_1)$.
\vskip .1cm

{\it Uniqueness:} Suppose $\chi^a$ and $\chi^b$ both satisfy the self-similarity equation \eqref{self_eqn}.  On a probability space, where $\{(\eta_j, X_j)\}_{j\geq 1}$, $\chi^a$ and $\chi^b$ are independent, define a sequence of measures:  $\chi^a_1=\chi^a$, $\chi^b_1=\chi^b$ and, for $j\geq 1$,
$$\chi^a_{j+1}=X_j\eta_j+(1-X_j)\chi^a_j \ \ {\rm and \ \ } \chi^b_{j+1}=X_j\eta_j+(1-X_j)\chi^b_j.$$
By construction, $\{\chi^a_j\}_{j\geq 1}$ and $\{\chi^b_j\}_{j\geq 1}$ are two sequences of identically distributed random measures distributed as $\chi^a$ and $\chi^b$ respectively.

We note again that $\prod_{j\geq 1}(1-X_j)=0$ a.s. as $\P(X=0)<1$.  Then, in terms of the variational norm $\|\cdot\|$,
\begin{align*}
&\left\| \chi^a_{j+1}- \chi^b_{j+1}\right\|  =\left|1-X_j\right|\left\| \chi^a_j-  \chi^b_j\right\|\\
 &\ \ \ \  =\left[\prod_{i=1}^j\left|1-X_i\right|\right]\left\|\chi^a_1- \chi^b_1\right\|
  \leq \prod_{i=1}^j\left|1-X_i\right|,
\end{align*}
which vanishes a.s. as $j\rightarrow\infty$.  Hence, 
$\chi^a\stackrel d=\chi^b$.
\qed

\medskip
\subsubsection{Completion of the proof of Theorem \ref{selfsimgem}}
\label{self_sim_comp}
Recall our conventions at the beginning of Section \ref{results_sect} and that ${\bf X}=\{X_j\}_{j\geq 1}$ is a collection of 
iid variables, and ${\bf T}$ is the homogeneous Markov chain with kernel $Q$ and initial distribution $\mu$ supported on recurrent states.  Let ${\bf P}= \langle P_j: j\geq 1\rangle$ be the RAM constructed from ${\bf X}$.  For each 
recurrent 
state $i$ of $Q$, let ${\bf T}^{i}={\bf T}\bigr|T_1=i$ be the Markov chain with transition kernel $Q$ and initial value $T^{i}_1=i$.
Recall the a.s. finite time 
$W^i=\inf\{l>1:T_l^i=i\}$, and variable 
\begin{align}
\label{ss_1}
X^i  &=\sum_{l=1}^{W^i-1}X_l\prod_{n=1}^{l-1}(1- X_n) = \sum_{l=1}^{W^i-1}P_l = 1-\prod_{l=1}^{W^i-1}(1-X_l).
\end{align}

Recall also
$
\eta^i  =\left(X^i\right)^{-1}\sum_{l=1}^{W^i-1}\left[ X_l\prod_{n=1}^{l-1}(1- X_n)\right]\delta_{T_l^i}$.
  
We now rewrite the measure 
$\nu^i = \nu\bigr|T_1=i$ as follows:
\begin{align}
\label{self_eqn1}
\nu^i &= \sum_{l\geq 1}P_l\delta_{T^i_l}\nonumber\\
&= \sum_{l=1}^{W^i-1}P_l\delta_{T^i_l} + \sum_{l\geq W^i} P_l \delta_{T^i_l} \nonumber\\ 
& = X^i\eta^i + \big(1-X^i\big)\sum_{l\geq W^i} \frac{P_l}{1-X^i}\delta_{T^i_l}.
\end{align}

Then, by \eqref{ss_1} and Proposition \ref{oneidentity_prop} for $j\geq 1$ we have
\begin{align*}
\frac{P_{j-1+W^i}}{1-X^i} &= \frac{X_{j-1+W^i}\prod_{l=1}^{j-1+W^i-1}(1-X_l)}{\prod_{l=1}^{W^i-1}(1-X_l)}\\
&= X_{j-1+W^i}\prod_{l=W^i}^{j-1+W^i-1}(1-X_l)\ = \ 
 X_{j-1+W^i}\prod_{l=1}^{j-1}(1-X_{l-1 + W^i}).
\end{align*}
Hence, as ${\bf X}$ is composed of iid variables, independent of ${\bf T}^i$ and therefore $W^i$, we see that
$$\big\langle \frac{P_{j-1+W^i}}{1-X^i} = X_{j-1+W^i}\prod_{l=1}^{j-1}(1-X_{l-1+W^i}): j\geq 1\big\rangle \stackrel{d}{=} \big\langle X_j\prod_{l=1}^{j-1}(1-X_l): j\geq 1\big\rangle = {\bf P}.$$
 
Clearly, as the chain starts over again at location $i$, $\{T^i_l\}_{l\geq W^i}\stackrel {d}{=} {\bf T}^i$.

Moreover, by conditioning on the value of $W^i$ and noting that ${\bf X}$ and ${\bf T}^i$ are independent, the sequences $\big\langle \frac{P_{j-1+W^i}}{1-X^i}: j\geq 1\big\rangle$ and $\{T^i_l\}_{l\geq W^i}$ are independent.  Similarly, we see that the sum $\sum_{l\geq W^i} \frac{P_l}{1-X^i}\delta_{T^i_l}$, which depends only on variables $\{X_k\}_{k\geq W^i}$ and $\{T^i_k\}_{k\geq W^i}$ indexed beyond the first cycle, is independent of the pair $(X^i, \eta^i)$.  In particular, the sum $\tilde\nu^i:=\sum_{l\geq W^i} \frac{P_l}{1-X^i}\delta_{T^i_l}\stackrel{d}{=} \nu^i$.

Hence, from these observations, \eqref{self_eqn1} represents the sought after self-similarity equation \eqref{self_similarity equation}.

Finally, a distribution $\nu^i$ satisfying \eqref{self_similarity equation} is unique by Lemma \ref{unique} since $X^i_1\in (0,1]$ a.s.  Also, by assumption, $T_1\sim \mu$
where $\mu$ is supported only on recurrent states.  Therefore, as $T_1$ necessarily is a recurrent state, the distribution of the pair $(\nu, T_1)$ is also unique. \qed

\subsection{Proof of Theorems \ref{momentchain}, \ref{moments_thm} and Corollary \ref{marginalmom}: Recasting moments I, II, and marginals}

We prove these results in succession.

\subsubsection{Proof of Theorem \ref{momentchain}}

First, since $G$ is a $k\times k$ generator matrix with bounded entries and for large enough $j\in\mathbb{N}$
$$
\widetilde K_j = \left(I -\frac{G}{j}\right)^{-1} 
= \sum_{n=0}^\infty \frac{G^n}{j^n},
$$
we verify that $\widetilde K_j = K_j + O(j^{-2})$

Next, to show \eqref{newocc}, we relate the occupation law of the Markov chain $\mathbf{Z}$, with transition kernels $\{\widetilde K_n\}$, to the occupation law $\nu$ of the Markov chain $\mathbf{T}$, with kernels $\{K_n\}$, through a Borel-Cantelli argument.  In passing, we note this could be also accomplished via an analytic argument.

Define $A_j:=\tilde K_j-K_j$, for $j\geq 1$, and note $A_j=O(j^{-2})$ has constant row sums of $0$. Since $G$ does not have $0$ entries and $K_j=I+\frac{G}{j}{\bf 1}(j>M)$, there exists an $a$ such that $R_j:=K_j+\frac{j^2}a A_j$ is a non-negative matrix, and hence stochastic. Note
\begin{equation}
\label{K j eq}
\tilde K_j=\left(1-\frac a{j^2}\right)K_j+\frac a{j^2}R_j.
\end{equation}

Consider now an auxilliary sequence of independent Bernoulli$(aj^{-2})$ variables ${\bf B}= \{B_j\}_{j\geq 1}$ by possibly enlargening the probability space.  Define a process ${\bf Z'}\bigr |{\bf B}$ with $Z'_1\bigr|{\bf B} \sim \mu$ and
\begin{align*}
 & \P\big(Z'_{j+1}=z_{j+1}\big |Z'_l=z_l:1\leq l\leq j, {\bf B}\big)\\
&\  =  \P\big(Z'_{j+1}=z_{j+1}\big |Z'_j=z_j, B_j\big) =  (1-B_j)K_j(z_j,z_{j+1})+B_jR_j(z_j,z_{j+1}).
\end{align*}
Then, noting \eqref{K j eq},
marginally, ${\bf Z}'$
is a Markov chain with initial distribution $\mu$ and transition kernel
\begin{align*}
&\P\big(Z'_{j+1}=z_{j+1}\big | Z'_j=z_j\big)\\
&\ \ \  = K_j(z_j, z_{j+1})\P(B_j = 0) + R_j(z_j, z_{j+1})\P(B_j=1) \\
&\ \ \  =  \tilde K_j.
\end{align*}

Now, by Borel Cantelli lemma, $\P(B_j=1\text{ i.o.})=0$ and so $L:=\max\{j:B_j=1\}<\infty$ a.s.  Conditional on the event that $\{L=r\}$, the chain $\{Z'_j\}_{j>r}$ is a Markov chain with transition kernels $\{K_j\}_{j>r}$.    
Also, since $G$ is irreducible in the setting of \cite{DS}, the initial distribution does not matter in the calculation of the occupation law $\nu$ (cf. Remark 3 in Subsection \ref{remarks subsection}).  Hence, the occupation law with respect to ${\bf Z}$ is also $\nu$ and \eqref{newocc} holds:  Indeed, for $l\in \X$ and interval $A=(a,b)$ for $0<a<b<1$, we have
\begin{align*}
&\lim_{n\rightarrow\infty}\P\Big( \frac{1}{n}\sum_{j=1}^n \mathbbm{1}(Z_j = l) \in A\Big)\\
&\ \  = \lim_{n\rightarrow\infty}\P\Big( \frac{1}{n}\sum_{j=1}^n \mathbbm{1}(Z'_j = l) \in A\Big)\\
&\ \ =\lim_{R\rightarrow\infty}\lim_{n\rightarrow\infty}\P\Big( \frac{1}{n}\sum_{j=1}^n \mathbbm{1}(Z'_j = l) \in A \ {\rm and \ } L<R\Big) + o(1)_R\\
&\ \ = \lim_{n\rightarrow\infty}\P\Big( \frac{1}{n}\sum_{j=1}^n \mathbbm{1}(T_j = l) \in A \Big) \ = \ \nu(l),
\end{align*}
where $o(1)_R$ is an expression which vanishes uniformly in $n$ as $R\rightarrow\infty$.

Finally, \eqref{newmoments} follows straightforwardly by gathering together terms.
\qed

\subsubsection{Proof of Theorem \ref{moments_thm}}
We break the argument into steps.
\vskip .1cm

\textit{Step 1.} 
First, we show that $p_j(\lambda)$, $q(j)$, and their quotients are all well-defined. 
A generator matrix $G$ can always be written as $G=\theta(Q-I)$ for some $\theta>0$ and a stochastic matrix $Q$.  The eigenvalues $\lambda$ of $Q$ correspond with the eigenvalues $\theta(\lambda-1)$ of $G$.  Additionally, since $G$ has no zero entries, $Q$ is irreducible.  Therefore, the algebraic multiplicity of the eigenvalue $0$ of $G$ is $1$. Thus, with respect to the minimal polynomial of $G$, $p_{min}(\lambda)$, there exists a polynomial $q$ such that $p_{min}(\lambda)=\lambda q(\lambda)$ and $q(0)\neq 0$.

Define 
$$\theta^G=\min\left\{\theta\in\mathbb{R}^+:I+G\theta^{-1}\text{ is non-negative}\right\}.
$$
  Since the eigenvalues of the stochastic matrix $I+G/\theta^G$ are bounded by $1$, the (complex) eigenvalues $\tilde\lambda$ of $G$ satisfy $\left|1+ \tilde\lambda/ \theta^G\right|\leq1$.  Hence, the eigenvalues of $G$ have non-positive real part.  Since $p_{min}(\lambda)=\lambda q(\lambda)$ and $q(0)\neq 0$, we obtain that $j\in \N$ is not an eigenvalue of $G$ and so $q(j)\neq 0$ for $j\geq 0$. Thus, $p_j(\lambda)/q(j)$ is well-defined for $j\geq 0$.
\vskip .1cm

\textit{Step 2.} We now verify for $j>0$ that 
$$\widetilde K_j = (I-G/j)^{-1}=\frac{p_j(G)}{q(j)}.$$
Write
\begin{align*}
(j-\lambda)p_j(\lambda) & =\sum_{i=0}^n\lambda^i\sum\limits_{l=i}^n a_l j^{l+1-i}-\sum_{i=0}^n\lambda^{i+1}\sum\limits_{l=i}^n a_l j^{l-i}\\
 & =\sum_{i=0}^n\lambda^i\sum\limits_{l=i}^n a_l j^{l+1-i}-\sum_{r=1}^{n+1}\lambda^r\sum\limits_{l=r-1}^n a_l j^{l+1-r}\\
 & =\sum_{l=0}^na_lj^{l+1}-\sum_{r=1}^{n+1}\lambda^r a_{r-1}\ =\ j q(j)-\lambda q(\lambda).
\end{align*}
In particular, as $Gq(G)=p_{min}(G)=0$, we have
$$I = \frac{jq(j)I-Gq(G)}{jq(j)} = \frac{(jI-G)p_j(G)}{jq(j)} = \left(I-\frac Gj\right)\frac{p_j(G)}{q(j)},$$
from which the desired identity follows.

\vskip .1cm
\textit{Step 3:} We now show that $p_0(G)/q(0)$ is the constant matrix with rows $\mu$.   Note that $p_0(\lambda)/q(0)=q(\lambda)/q(0)$ is well-defined in \eqref{q equation}.  Since row sums of $G^k$ vanish for $k\geq 1$, we see that $p_j(G)/q(j)$ has constant row sums of $p_j(0)/q(j)=1$.  Now, necessarily, $p_0(G)G=0$ as $q(\lambda)\lambda = p_{min}(\lambda)$ is the minimal polynomial of $G$.  Since $G$ is irreducible, we can conclude that $p_0(G)$ is a matrix with rows given by multiples of the unique stochastic eigenvector $\mu$ associated to $G$ and eigenvalue $0$.  However, since $p_0(G)/q(0)$ has row sums equal to $1$, the claim follows.

Moreover, noting that $[p_0(G)/q(0)]_{i,j} = \mu_j$ for any $i,j\in \X$, the moment identity \eqref{alg_mom} is now a direct consequence of these calculations. \qed

\subsubsection{Proof of Corollary \ref{marginalmom}}
Recall $q(\lambda) = \sum_{i=0}^n a_i \lambda^i$ is a degree $n\leq k-1$ polynomial where $a_n\neq 0$.  Then, noting \eqref{q equation}, we see that $p_j(\lambda)$ is also degree $n$ polynomial in $j$ with $\lambda$-free leading coefficient $a_n$. In particular, $[p_j(G)]_{i,i}$ is a degree $n$ polynomial in $j$ with leading coefficient $a_n I_{i,i}=a_n$ for each $i\in\mathcal{X}$. 

Now, fix $i$, and denote by $\{\gamma_{i,l}\}_{l=1}^n$ and $\{\lambda_l\}_{l=1}^n$ the roots of $[p_j(G)]_{ii}$ and $q(j)$ respectively when considered as functions of $j$. In the formula \eqref{alg_mom}, to calculate $\E[\nu_i^N]$, there is only one list in ${\mathbb S}(N)$, namely one composed of $N$ $1$'s.  Then,
\begin{align*}
&\E\left[\nu_i^N\right]  =\prod_{j=0}^{N-1}\left[\frac{p_j(G)}{q(j)}\right]_{i,i}=\prod_{j=0}^{N-1}\frac{\left[p_j(G)\right]_{i,i}}{q(j)}\\
&\ \ =\frac{\prod_{j=0}^{N-1}\prod_{l=1}^n(j-\gamma_{i,l})}{\prod_{j=0}^{N-1}\prod_{l=1}^n(j-\lambda_l)}=\frac{\prod_{l=1}^n\prod_{j=0}^{N-1}(-\gamma_{i,l}+j)}{\prod_{l=1}^n\prod_{j=0}^{N-1}(-\lambda_l+j)}  = \prod_{l=1}^n\frac{\left(-\gamma_{i,l}\right)_N}{\left(-\lambda_l\right)_N},
\end{align*}
as desired. \qed

\medskip
{\bf Acknowledgement.} 
We thank J. Sethuraman for enjoyable conversations on Dirichlet processes.
Part of this research was supported by ARO W911NF-14-1-0179, and a Simons Foundation Sabbatical grant.

\end{document}